\documentclass[11pt,reqno]{amsart}
\usepackage{color}
\usepackage[active]{srcltx}
\usepackage{a4wide}
\usepackage{amssymb, amsmath, mathtools}
\usepackage{graphicx}
\usepackage{float}
\usepackage[active]{srcltx}
\usepackage[ansinew]{inputenc}
\usepackage{hyperref}

\theoremstyle{plain}
\newtheorem{theorem}{Theorem}[section]
\newtheorem{maintheorem}{Theorem}
\newtheorem{lemma}[theorem]{Lemma}
\newtheorem{proposition}[theorem]{Proposition}
\newtheorem{corollary}[theorem]{Corollary}
\newtheorem{maincorollary}[maintheorem]{Corollary}

\theoremstyle{remark}

\newtheorem{definition}{Definition}

\newtheorem{remark}[theorem]{Remark}

\numberwithin{equation}{section}

\newcommand{\NN}{{\mathbb{N}}}
\newcommand{\ZZ}{{\mathbb{Z}}}
\newcommand{\RR}{{\mathbb{R}}}
\newcommand{\EU}{{\mathbb{S}}}

\newcommand{\In}{{\text{In}}}
\newcommand{\Out}{{\text{Out}}}
\newcommand{\Fix}{{\text{Fix}}}
\newcommand{\loc}{{\text{loc}}}

\newcommand{\dpt}{\displaystyle}

\begin{document}

\title[Dissecting a resonance wedge on heteroclinic bifurcations]{Dissecting a resonance wedge on \\  Heteroclinic bifurcations}
\author[Alexandre Rodrigues]{Alexandre A. P. Rodrigues \\ Centro de Matem\'atica da Univ. do Porto \\ Rua do Campo Alegre, 687,  4169-007 Porto,  Portugal }
\address{Alexandre Rodrigues \\ Centro de Matem\'atica da Univ. do Porto \\ Rua do Campo Alegre, 687 \\ 4169-007 Porto \\ Portugal}
\email{alexandre.rodrigues@fc.up.pt}

\date{\today}

\thanks{AR was partially supported by CMUP (UID/MAT/00144/2019), which is funded by FCT with national (MCTES) and European structural funds through the programs FEDER, under the partnership agreement PT2020. AR also acknowledges financial support from Program INVESTIGADOR FCT (IF/00107/2015).}


\begin{abstract}

This article studies  routes to chaos  occurring within a resonance wedge for a 3-parametric family of differential equations acting on a 3-sphere. 
Our starting point is an autonomous vector field whose flow exhibits a weakly attracting  heteroclinic network made by two 1-dimensional connections and a 2-dimensional separatrix between two  equilibria with different Morse indices.  After changing the parameters, while keeping the 1-dimensional connections unaltered, we concentrate our study in the case where the 2-dimensional invariant manifolds of the equilibria do not intersect. 

 We derive the first return map near the network and we reduce the analysis of the system to a 2-dimensional map on the cylinder. 
Complex dynamical features arise from a discrete-time Bogdanov-Takens singularity, which  may be seen as the organizing center by which one can obtain infinitely many  attracting tori, strange attractors, infinitely many sinks and non-trivial contracting wandering domains.  
These dynamical phenomena occur within a structure that we call \emph{resonance wedge}. As an application, we may see the ``classical'' Arnold tongue as a projection of a resonance wedge.  
The results are general, extend to other contexts and lead to a fine-tuning of the theory.

\end{abstract}

\maketitle
\setcounter{tocdepth}{1}

\section{Introduction}\label{intro}

To date there has been very little systematic investigation of the effects of  perturbations that break an invariant torus, despite being natural for the modelling of many biological and physical effects  \cite{AHL2001, Kirk93, KR2008, Langford, Rodrigues2, SNN95}. 
In this paper, we describe the transition from regular dynamics to chaos associated to the \emph{Torus-breakdown Theory} partially described  in Afraimovich and Shilnikov \cite{AS91}, applied to a specific heteroclinic configuration 
involving 2-dimensional connecting manifolds (\emph{continuum of connections} in the terminology of   \cite{AC}).

 In studying bifurcations associated to \emph{Torus-breakdown},  it is natural to examine the bifurcation diagram in terms of \emph{Arnold tongues} \cite{Arnold65}.
  In  planar maps,  a $\frac{p}{q}$--resonance tongue  can be defined as the locus, in the parameter space, where periodic points with rotation number $\frac{p}{q}$ exist,  for $p, q \in \NN$; in the engineering literature, these solutions are also called ${p}:{q}$ phase-locked \cite{Denjoy, Herman}.
       The way in which resonance tongues overlap and evolve indicates how the rotational dynamics is changing.

     Although tongues corresponding to different rotation numbers  can overlap, in general there is no path where a periodic orbit from a resonance tongue becomes a periodic orbit from another  tongue   \cite{Boyland, Denjoy}.  A singular phenomenon has been described by Kirk \cite{Kirk93}, who studied resonance zones for a 3-dimensional system that corresponds to the normal form associated to  a codimension-two Hopf-zero singularity.  Such a normal form is perturbed with nonsymmetric terms 
 breaking the axial symmetry and the phenomenon  of \emph{``merging of resonance wedges''} has been described: far from the torus bifurcation,  a periodic orbit  modifies its shape and collapses with another periodic orbit coming from another tongue. 

\medbreak

Assuming a strong 1-dimensional contracting direction, the structure of the Arnold tongue has been given by Arnold \cite{Arnold65}, Boyland \cite{Boyland} and Herman \cite{Herman} who reduced the study of \emph{Torus-breakdown} effects to the ``canonical family'' on the circle:
$$
x \quad \mapsto \quad x+a+\frac{b}{2\pi}\sin(2\pi x) \qquad \mod{1}, \qquad a, b\in \RR.
$$

 When the radial contracting foliation on the torus is lost, the bifurcation structure of the circle maps family is inadequate to explain the diverse phenomena that  accompany the loss of the attracting torus and it is here that the study of resonance wedges plays an important role.  Resonance tongues are associated to a wide range of behaviours such us: the existence of quasi-periodic solutions, sinks, saddle-node bifurcations, homoclinic orbits, bistability, rotational horseshoes and strange attractors (either ``large'' or ``small'' according to  Broer \emph{et al} \cite{BST98}).   Early papers in this context are \cite{AS91, AHL2001, Anishchenko, Aronson}.   We address the reader to \cite{AH2002, AP93, Boyland, BST98, GreenspanHolmes, Peckman_bananas} for more information on the subject. For a tutorial,  see Shilnikov \emph{et al} \cite{Shilnikov_tutorial}.  
New directions of the  theory and applications to periodiocally-kicked differential equations  can be found in \cite{Bakri, Bakri2,  CFM2020,  Rodrigues2019, WY}.

The goal of this paper is to construct a 3-parameter bifurcation diagram for a concrete configuration associated to a weakly attracting heteroclinic network  with two saddle-foci previously studied in \cite{LR2016, Rodrigues2019},  the \emph{Bykov attractor}.  
Our  study has been motivated by numerical results obtained in   \cite{Algaba2001, CFM2020, CastroR2019}.

Our purpose in writing this paper is not only to point out the range of phenomena that can occur when an invariant torus is broken, but to bring to the foreground the techniques that have allowed us to reach these conclusions in a relatively straightforward manner. 
These mechanisms are not limited to the heteroclinic network considered here.

\subsection{The novelty.}
While some progress has been made, both numerically and analytically, the number of explicit configurations whose flows have an invariant torus and for which the \emph{Torus-Breakdown} description is available, has remained small.

By studying unfoldings of a Bykov attractor,  we are able to delineate the ways in which the first return map to a cross section  can make the transition from a single rotation number to an interval of rotation numbers\footnote{The existence of \emph{rotational horseshoes}, described in Appendix \ref{Rotational horseshoe},  is responsible for the existence of an interval of rotation number -- see \cite{PPS}. }.  As a continuation of the project started in \cite{Rodrigues2019}, we will prove analytically that a sequence of discrete-time Bogdanov-Takens bifurcations organise the dynamics that appear in the unfolding.   
Besides, within the Arnold wedge, we numerically find surfaces corresponding to the following bifurcations: Hopf, period-doubling and the transition \emph{node $\mapsto$ focus}.    Our results agree well with the literature about heteroclinic bifurcations, Arnold tongues and the  scenarios described in  \cite{AP93, Aronson, Peckman_bananas}.
 
Our object of study  is not ``\emph{just another dynamical system}'', but   representative for the case of 3-dimensional dissipative flows admitting 2-dimensional connections that are pulled apart.  In similar models, the corresponding phenomenology should contain no further secrets. 

\subsection{Physical setting.}
Our study allows us to understand the bifurcations from an invariant torus to strange attractors  that appear in Ruelle and Takens  \cite{RT71} and Langford \cite{Langford} (see also \S 6.2 of  \cite{CFM2020}).   In the context of turbulent flows, the author of \cite{Langford}  studied a two-parameter unfolding a Hopf-zero singularity and proved that axisymmetric perturbations generate an invariant torus.
By slightly breaking the symmetry,  Langford prove  that the flow becomes more and more turbulent with fractal basins of attraction as a consequence of the emergence strange attractors.
 The bifurcations described in our paper have similarities with those described in   \cite{SNN95}  in the context of a low-order atmospheric circulation model.

\subsection{The structure.}

 In Section \ref{s:setting}, we describe precisely our object of study and we review the literature related to it. In Section \ref{main results} we state the main results of the article. The coordinates and other notation used in the rest of the article are presented in Section \ref{localdyn} to prove the main results of the manuscript in  Section \ref{proof_th_A}.

We  refine the structure of Arnold tongues which appear in the context of heteroclinic bifurcations on Section \ref{dissection}.
 In Section \ref{s:example}, we briefly illustrate our theoretical results with an example explored in \cite{CastroR2019}. Symmetry-breaking effects will be described.

We finish the article with a discussion in Section \ref{discussion} about  the consequences of our findings. Dynamics similar to what we described is expected to occur   near periodically forced  attracting heteroclinic cycles.  For reader's convenience, we have compiled at  the end of the article (Appendix \ref{Definitions}) a list of basic  definitions. 

\section{Setting and state of art}
\label{s:setting}
In this section, we describe the main hypotheses about the weakly heteroclinic network we are considering. We postpone the technical definitions  used in this section to Appendix \ref{Definitions}.

\subsection{Starting point}
\label{starting point}
For $\varepsilon>0$ small, consider the 3-parameter family of $C^3$--smooth differential equations
\begin{equation}
\label{general2.1}
\dot{x}=f_{(A, \lambda, \omega)}(x)\qquad x\in \EU^3\subset \RR^4 \qquad A, \lambda \in [0, \varepsilon],\qquad  \omega \in \RR^+
\end{equation}
 where $\EU^3$ represents the unit three-sphere, endowed with the usual topology. Let us denote by $\varphi_{(A, \lambda, \omega)}(t,x)$, $t \in \RR$, the   flow associated to \eqref{general2.1} 
 satisfying the following properties for $A=\lambda=0$ and $\omega \in \RR^+$:

\bigbreak
\begin{enumerate}
 \item[\textbf{(P1)}] \label{B1}  There are two different equilibria, say $O_1$ and $O_2$.
 \bigbreak
 \item[\textbf{(P2)}] \label{B2} The eigenvalues of $Df_X$ are:
 \medbreak
 \begin{enumerate}
 \item[\textbf{(P2a)}] $E_1$ and $ -C_1\pm \omega i $ where $C_1,E_1>0$, for  $X=O_1$;
 \medbreak
 \item[\textbf{(P2b)}] $-C_2$ and $ E_2\pm \omega i $ where $C_2, E_2>0$,   for $X=O_2$.
 \end{enumerate}
 \end{enumerate}
\bigbreak 
  For $W\subseteq \EU^3$, denoting by $\overline{W}$ the  closure of $W$, we also assume that:
 \begin{enumerate}
 \bigbreak
  \item[\textbf{(P3)}]\label{B3} The manifolds $\overline{W^u(O_2)}$ and $\overline{ W^s(O_1)}$ coincide and  the set $\overline{W^u(O_2)\cap W^s(O_1)}$ consists of a 2-dimensional sphere ($ {W^u(O_2)\cap W^s(O_1)}$ is called the $2D$-connection).
  \end{enumerate}\bigbreak
  and 
   \bigbreak
   \begin{enumerate}
\item[\textbf{(P4)}]\label{B4} There are two trajectories, say  $\gamma_1, \gamma_2$, contained in  $W^u(O_1)\cap W^s(O_2)$, each one in each connected component of $\EU^3\backslash \overline{W^u(O_2)}$ ($\gamma_1, \gamma_2$ are called the $1D$-connections).
\end{enumerate}
 
 \begin{figure}[h]
\begin{center}
\includegraphics[height=4cm]{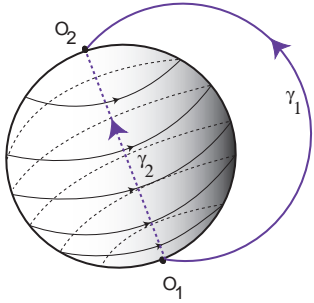}
\end{center}
\caption{\small  Scheme of the attractor $\Gamma$ satisfying \textbf{(P1)--(P4)}, for $A=\lambda=0$ and $\omega>0$. }
\label{Bykov1}
\end{figure}

\bigbreak
The equilibria $O_1$ and $O_2$, the $2D$-connection referred in \textbf{(P3)} and the two trajectories $\gamma_1, \gamma_2$ of  \textbf{(P4)}  build a heteroclinic network, that we denote by $\Gamma$. This network, illustrated in Figure \ref{Bykov1},  has two cycles.   Assuming that the set of eigenvalues of $Df_{O_1}$ and $Df_{O_2}$  satisfy 
   \medbreak 
\begin{enumerate}
 \item[\textbf{(P5)}] \label{P5}  $\dpt\frac{C_1 C_2}{ E_1 E_2} \gtrsim 1$
  \end{enumerate}
\medbreak 
 \noindent the network $\Gamma$ is asymptotically stable (proof in \cite{LR2016}): with exception of the origin, all trajectories are forward attracted to $\Gamma$.   As a consequence, we may find a neighborhood $\mathcal{U}$ of   $\Gamma$ having its boundary transverse to the flow  and such that every solution starting in $\mathcal{U}$ is asymptotic to $\Gamma$. The set $\Gamma$ is usually called \emph{Bykov attractor}\footnote{The terminology \emph{Bykov} is a homage to V. Bykov who dedicated his latest research projects to similar cycles \cite{Bykov99, Bykov00}.}.
\subsection{Chirality: a topological assumption}
There are two  possibilities for the geometry of the flow around the saddle-foci of $\Gamma$,
depending on the direction the trajectories turn around the  $1D$-connections.
This is related to the topological concept of \emph{chirality} introduced in \cite{LR2015}. 
\medbreak
Let $V_1$
and $V_2$ be small disjoint neighborhoods of $O_1$ and $O_2$ with boundaries $\partial V_1$ and $\partial V_2$, respectively. These neighborhoods will be precisely constructed   in Section~\ref{localdyn}.
Trajectories starting at $\partial V_1\backslash W^s(O_1)$ near $W^s(O_1)$ go into the interior of $V_1$ in positive time, then follow one of the solutions in $[O_1 \rightarrow O_2]$, go inside $V_2$,  come out at $\partial V_2$ and then return to $\partial V_1$ (see Figure~\ref{Chirality}). This trajectory is not closed since $\Gamma$ is attracting. 

Let $\mathcal{Q}$ be a piece of trajectory like this from $\partial V_1$ to $\partial V_1$. Within $\partial V_1\backslash W^s(O_1)$, join its starting point to its end point by a segment as in Figure~\ref{Chirality}, forming a closed curve, which we call the  \emph{loop} of $\mathcal{Q}$. 
By construction, the loop of $\mathcal{Q}$ and  $\Gamma$ are disjoint closed sets. 

\begin{definition}\cite{LR2015}
We say that  $O_1$ and $O_2$ in $\Gamma$ have the  \emph{same chirality} if the loop of every trajectory starting near $O_1$ is linked to $\Gamma$ (\emph{i.e} the trajectories cannot be disconnected by an isotopy). Otherwise, we say that $O_1$ and $O_2$ have \emph{different chirality}.
\end{definition}
 The next hypothesis may be written as:

\medbreak

\begin{enumerate}
\item[\textbf{(P6)}] \label{B6} The saddle-foci $O_1$ and $O_2$ have the same chirality.
\end{enumerate}

\begin{figure}[h]
\begin{center}
\includegraphics[height=3cm]{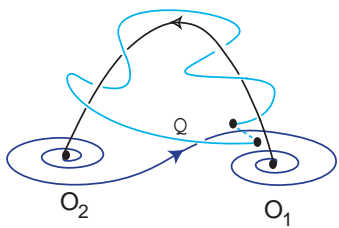}
\end{center}
\caption{\small  Illustration of Property  \textbf{(P6)}: the saddle-foci $O_1$ and $O_2$ have the same chirality. }
\label{Chirality}
\end{figure}

For  $r \geq 3$, denote by  $\mathfrak{X}^r(\EU^3)$, the set of  3-parameter family of $C^3$--vector fields on $\EU^3$ satisfying Properties \textbf{(P1)--(P6)}, endowed with the $C^r$--topology.\bigbreak
\subsection{Perturbing terms}
Concerning the effect of  $A$, $\lambda$ and $\omega$ on the dynamics of \eqref{general2.1}, we assume that:
 \medbreak
\begin{enumerate}
\item[\textbf{(P7)}] \label{B6} For $A> \lambda \geq 0$ and $\omega \in \RR^+$, the two trajectories within $W^u(O_1)\cap W^s(O_2)$ persist.
\end{enumerate}
 \medbreak
By the Kupka-Smale Theorem, generically the invariant 2-dimensional manifolds $W^u (O_2)$ and $W^s (O_1)$ are  transverse (either intersecting or not): 
 \medbreak
\begin{enumerate}
\item[\textbf{(P8a)}]\label{B8} For $A, \lambda \geq 0$ and $\omega \in \RR^+$, the  manifolds $W^u(O_2)$ and $W^s(O_1)$ intersect transversely. 
\medbreak
\item[\textbf{(P8b)}]\label{B8b} For $A, \lambda \geq 0$ and $\omega \in \RR^+$, the  manifolds $W^u(O_2)$ and $W^s(O_1)$ do not intersect.
\end{enumerate}

 \medbreak

 and
\medbreak

\begin{enumerate}
\item[\textbf{(P9)}] \label{B9} Up to high order terms in $x,y$, the transitions along the connections $[O_1 \rightarrow  O_2]$ and $[O_2 \rightarrow  O_1]$ are given, in the local coordinates that will be defined in Section~\ref{localdyn}, by the \emph{Identity map} and by  $$(x,y)\mapsto (x,y+A + \lambda \Phi(x))$$ respectively, where $\Phi:\EU^1 \rightarrow \EU^1$ is a Morse function with at least two non-degenerate critical points ($\EU^1=\RR \pmod{2\pi}$). This assumption will be clearer later.
\end{enumerate}
\medbreak

\begin{figure}[h]
\begin{center}
\includegraphics[height=4.1cm]{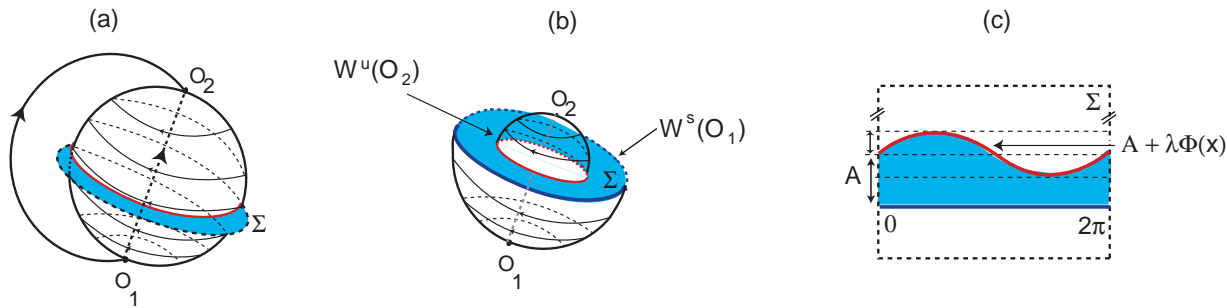}
\end{center}
\caption{\small  Illustration of Properties \textbf{(P8b)} and \textbf{(P9)}, where $\Sigma$ represents a global cross section to $\Gamma$. Double bars mean that the sides are identified.}
\label{scheme1} 
\end{figure}

\subsection{Constants}
\label{constants1}
Once for all, we define the following notation that will be used throughout the present manuscript:

\begin{eqnarray}
\label{constants}
\delta_1 = \frac{C_1}{E_1 }>0 \qquad & \dpt  \delta_2 = \frac{C_2}{E_2 }>0& \qquad \delta=\delta_1\, \delta_2\gtrsim 1\nonumber \\ \nonumber \\  \nonumber \\ 
\omega^\star_\ell = \frac{2\ell \pi (\delta-1)}{\ln \delta}>0 \qquad  & \dpt K = \frac{E_2 +C_1  }{E_1E_2}>0  & \qquad \dpt  
 \qquad \qquad \\ \nonumber \\  \nonumber \\ 
\dpt M= \dpt \delta ^{\frac{1}{1-\delta}}- \delta ^{\frac{\delta}{1-\delta}}  \qquad &\dpt  \mu = (A, \lambda, \omega)&\qquad  \nu = (A, \lambda_0, \omega), \, \, \,\lambda_0>0 \text{  fixed} \nonumber  \nonumber \\ \nonumber
\end{eqnarray}

\subsection{Digestive remarks about the hypotheses}
We discuss the Hypotheses  \textbf{(P1)--(P9)}, stressing  that they are natural in several settings. An illustrative scheme of the  Hypotheses  has been summarized in Table 1  of \cite{Rodrigues2019}.
\begin{remark}
\label{simetria_remark}
Although the fully description of the bifurcations associated to the heteroclinic attractor $\Gamma$ is a phenomenon of codimension three \cite{KLW, LR},  the setting described by \textbf{(P1)--(P9)}  is natural in symmetric contexts \cite{Aguiar_tese, LR2016} and also in some unfoldings of the Hopf-zero singularity \cite{BIS, Gaspard, SNN95}.  
\end{remark}

\begin{remark}
Hypothesis \textbf{(P7)} corresponds to the \emph{partial symmetry-breaking} considered in Section 2.4 of \cite{LR}.
The setting described by \textbf{(P1)--(P8b)} and \textbf{(P9)} generalizes Case (4) of~\cite{Rodrigues2019}. 
\end{remark}

\begin{remark}
The first hit of $W^u(O_2)$ and $W^s(O_1)$ to a global cross section $\Sigma$ are two closed curves (see Figure \ref{scheme1}); the distance between these two curves  can be written as $$A + \lambda \Phi(x),\qquad  x\in \EU^1,$$  which may be seen as an  approximation of the \emph{Melnikov integral} associated to the intersection of $W^u(O_2)$ and $W^s(O_1)$  \cite[\S 4.5]{GH}.  
 \end{remark} 

\begin{remark}
Variable $\omega$ represents the speed of rotation of the saddle-foci.  Using different imaginary parts on the complex eigenvalues of $Df_{O_1}$ and $Df_{O_2}$   would complicate the expression of $K \omega$, without any qualitative benefits in the final result. 
\end{remark}

\begin{remark}
Derivations using a more general form for the transition $[O_1 \rightarrow O_2]$ have been performed in  Section 6 of \cite{Rodrigues3}.
The transition along $[O_2 \rightarrow O_1]$  corresponds to the expected unfolding from the coincidence of the 2-dimensional invariant manifolds at $f_{(0,0,\omega)}$, $\omega\in \RR^+$.

\end{remark}

\begin{remark}
\label{P7-P8}
The technical Hypothesis \textbf{(P9)} modulates the two generic possibilities given by \textbf{(P8a)} and \textbf{(P8b)}:

 \begin{eqnarray*}
\lambda>A \geq 0, \,\, \,  \omega \in \RR^+ \quad &\Leftrightarrow& \quad \text{\textbf{(P8a)}: $W^u(O_2)$ and $W^s(O_1)$ intersect transversely;} \\ \\
A>\lambda \geq 0,  \,\, \,  \omega \in \RR^+ \quad &\Leftrightarrow & \quad \text{\textbf{(P8b)}: $W^u(O_2)$ and $W^s(O_1)$ do not intersect.}\\
\end{eqnarray*}
Assumption \textbf{(P9)}  governs the transition maps along the heteroclinic connections and is necessary to make precise computations. All results are valid for any $2\pi$-periodic non-constant Morse function defined on $[-1,1]$.    \\
\end{remark}

\subsection{State of art}
In this section, we give an overview of results for the class of vector fields satisfying either \textbf{(P1)--(P8a)--(P9)} or \textbf{(P1)--(P8b)--(P9)}.

\subsubsection{Heteroclinic tangle:} 
If $f_{(A, \lambda, \omega)} \in\mathfrak{X}^3(\EU^3)$ and satisfies \textbf{(P7)--(P8a)--(P9)}, then the 2-dimen\-sional invariant manifolds $W^u(O_2)$ and $W^s(O_1)$ meet transversely, giving rise to a union of Bykov cycles \cite{Bykov00, LR2016}. The dynamics in the maximal invariant set inside $\mathcal{U}$, contains the suspension of horseshoes accumulating on the   network described in \cite{ACL05, KLW, LR, Rodrigues3}. Near the original heteroclinic attractor $\Gamma$, the flow contains infinitely many homoclinic tangencies and sinks with long  periods, coexisting with sets with positive entropy, giving rise the so called \emph{quasi-stochastic attractors} \cite{Gonchenko97}; more details in Appendices \ref{quasi1} and \ref{app: HSt}.

\subsubsection{Torus-breakdown}
\label{ss: transitional}
 If $f_{(A, \lambda, \omega)} \in\mathfrak{X}^3(\EU^3)$ and satisfies \textbf{(P7)--(P8b)--(P9)}, then the 2-dimen\-sional invariant manifolds  of the saddle-foci do not intersect. According to \cite{Rodrigues2019}, in the bifurcation diagram $\left( \omega,  {\lambda} / {A}\right)$ of Figure \ref{graph1}, there are two curves, the graphs of $h_1$ and $h_2$,  such that,  for all $\omega\in \RR^+$, we have  $h_1(\omega)<h_2(\omega)$ and:
\begin{enumerate}
\medbreak
\item \textbf{Attracting torus:} the region below the graph of $h_1$  corresponds to parameters whose flow has an attracting normally hyperbolic 2-torus with zero topological entropy -- Theorem B of \cite{Rodrigues2019}. In the bifurcation parameter $(\omega, \lambda/A)$, there exists a set of \emph{positive Lebesgue measure}  for which the whole torus is the minimal attractor.  
\begin{figure}[h]
\begin{center}
\includegraphics[height=5.5cm]{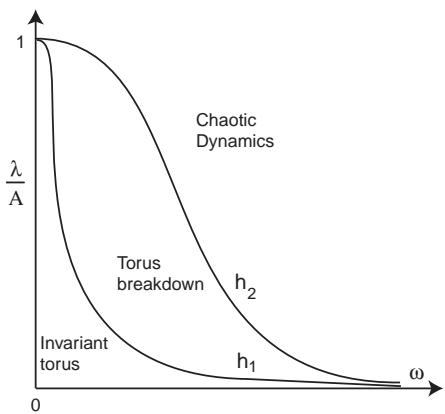}
\end{center}
\caption{\small Summary of the results of \cite{Rodrigues2019}: location of the attracting 2-torus and chaotic regions with respect to   $\omega$ and $\frac{\lambda}{A}$, for a vector field  $f_{(A, \lambda, \omega)} \in\mathfrak{X}^3_{\text{Byk}}(\EU^3)$.  }
\label{graph1}
\end{figure}
\medbreak
\item\textbf{Transitional dynamics:}  for a fixed $\omega>0$, in the transition from $h_1(\omega)$ to $h_2(\omega)$,  the attracting torus breaks. It  starts to disintegrate into a finite collection of periodic saddles and sinks, a phenomenon occurring within an \emph{Arnold tongue} \cite{Rodrigues2019,  Shilnikov_tutorial, WY}.  Each time the Floquet multipliers of periodic orbits cross a root of unity, a pair of saddle-node bifurcation curves may be defined. These curves limit locally the corresponding resonance tongue. In addition, there are regions corresponding to homoclinic tangencies to dissipative periodic solutions, responsible for the \emph{persistence of H\'enon-like strange attractors} \cite{Rodrigues2019}. 
\medbreak.
\item \textbf{Chaotic dynamics:} the region above the graph of $h_2$ corresponds to vector fields whose flows exhibit \emph{rotational horseshoes}  \cite{Rodrigues2019} (see Appendix \ref{Rotational horseshoe}). Theorem D of \cite{Rodrigues2019} may be seen as a criterion to obtain rotational horseshoes near $\Gamma$; once they develop, they persist for small perturbations.

\medbreak

\end{enumerate}
\bigbreak
 
As in \cite{AHL2001}, the graphs of $h_1$ and $h_2$ are not bifurcation lines; they   define a region inside which  the transitional dynamics occurs. 
From now on, without loss of generality,  let us also assume that  $\Phi(x)=\sin x$,  $x\in \EU^1$, which has exactly two critical points. 
It simplifies the computations and allows comparison with previous works. 
  For $r \geq 3$, we denote by $\mathfrak{X}^r_{\text{Byk}}(\EU^3)\subset \mathfrak{X}^r(\EU^3)$, the set of $C^r$--vector fields on $\EU^3$ satisfying conditions  \textbf{(P1)}--\textbf{(P8b)} and \textbf{(P9)}.

\section{Main results}\label{main results}

Let $\mathcal{T}$ be a neighborhood of the heteroclinic attractor $\Gamma$, which exists for $A=\lambda=0$ and $\omega \in \RR^+$.  For $\varepsilon>0$ small, 
define the set
\begin{equation}
\label{paramater_set}
\mathcal{V}= \{\mu=(A, \lambda, \omega):  \qquad 0\leq \lambda< A \leq \varepsilon \qquad M\geq  A+\lambda \qquad  \text{and} \qquad \omega \in \RR^+ \}
\end{equation}\bigbreak
\noindent
and let $\left(f_{\mu}\right)_{\mu \in \mathcal{V}}$ be a 3-parameter family of vector fields in $\mathfrak{X}_{\text{Byk}}^3(\EU^3)$. 
According to  \cite{Rodrigues2019}, there is $\tilde\varepsilon>0$ such that the first return map to a given global cross section $\Sigma$ to $\Gamma$ can be expressed, in local coordinates $(x,y)\in \Sigma$,   by:
\medbreak
\begin{equation}
\label{first_return_map}
\mathcal{F}_{\mu}(x, y)=\left[ x - K \, \omega  \ln ( y+A+\lambda\sin x) \pmod{2\pi}  , \, \, ( y+A+\lambda\sin x)^\delta\right] + \ldots
\end{equation}
where $$(x,y)\in \mathcal{D}=\{x\in \RR \pmod{2\pi}, \quad y/\tilde \varepsilon  \in [-1, 1] \quad \text{and} \quad y + A + \lambda \sin x> 0\}$$ and $\dots$ stand for   small  terms depending on $x$ and $y$  converging to zero along their derivatives. The main steps to get the expression \eqref{first_return_map} are revived in Section \ref{localdyn}.
 Since $\delta \gtrsim 1$, for $A>0$ sufficiently small, the second component of $\mathcal{F}_{\mu}$ is contracting in $y$ (Remark \ref{rem_dissipative}).
 \bigbreak
 \begin{definition}
Let $(x_0, y_0)\in \mathcal{D}$ and $\ell \in \NN$. We say that $(x_0, y_0)$ is a $(1,\ell)$--fixed point of $\mathcal{F}_{\mu}$ if $\mathcal{F}_{\mu}(x_0, y_0)= (x_0+2\ell \pi, y_0)$. \end{definition}
  \medbreak 

  The main contribution of this article is the following result:
  \medbreak 

\begin{maintheorem}\label{thm:0}
Let $f_{\mu} \in\mathfrak{X}_\emph{Byk}^3(\EU^3)$ and $\ell \in \NN$.
There are two curves in the 3-dimensional parameter space $\mu=(A, \lambda, \omega)\in \mathcal{V}$ where $(1,\ell)$--fixed points of $\mathcal{F}_{\mu}$ undergo a discrete-time Bogdanov-Takens bifurcation. The points in the curves occur at values of $\omega= \omega^\star_\ell$.
\end{maintheorem}

The proof of Theorem \ref{thm:0} is performed in Subsection \ref{prova_pd}, after the statement of some preliminary technical results.
Around the Bogdanov-Takens bifurcation, we found the following (secondary) codimension one bifurcations:

\begin{maincorollary}
\label{Corol1}
Let $f_{\mu} \in\mathfrak{X}_\emph{Byk}^3(\EU^3)$.
There are surfaces in the 3-dimensional parameter space $\mu=(A, \lambda, \omega)\in \mathcal{V}$ where $(1,\ell)$--fixed points of $\mathcal{F}_{\mu}$ undergo:\medbreak
\begin{enumerate}
\renewcommand{\theenumi}{(\alph{enumi})}
\renewcommand{\labelenumi}{{\theenumi}}
\item\label{saddlenodes}  a saddle-node bifurcation; \\
\item\label{Hopf} a Hopf bifurcation;\\
\item\label{homoclinic} generic (quadratic) homoclinic tangencies associated to a dissipative saddle point\footnote{A  $\mathcal{F}_{\mu}$-fixed point $O$ is \emph{dissipative} if $O$ is hyperbolic and  $|\det \mathcal{F}_{\mu}(O)| <1$.}. \medbreak
\end{enumerate}
In the region between the surfaces corresponding to homoclinic tangencies, there is a transverse intersection of the stable and the unstable manifolds of an invariant saddle. The surfaces corresponding to Hopf and saddle-node bifurcations meet tangentially along a curve.
 \end{maincorollary}

 Around the transverse intersection of the manifolds, horseshoe dynamics occurs.  Corollary \ref{Corol1} refines the findings of Sections 3.2 and 3.3 of  \cite{Kim}, where the resonance wedges are projection of bifurcation surfaces. See also   \cite{PK2002}.
The `necessity' of a Hopf bifurcation surface to explain the transition from an invariant torus to chaos has been raised in \cite{Peckman90}.  
 
Strange attractors contribute to the richness and complexity of a dynamical system. Sinai-Bowen-Ruelle (SRB) measures represent visible statistical laws in non-uniformly hyperbolic systems. More details of these concepts may  be found  in Appendices \ref{ss: strange attractor} and \ref{ss: SRB measure}. Chaos associated with them is both \emph{sustained in the space of parameters} and \emph{observable}.   Next result ensures the existence of open regions in the parameter region \eqref{paramater_set} for which we observe strange attractors with SRB measures and historic behaviour (see Appendix \ref{ss: historic behaviour}):

\begin{maincorollary}
\label{main_thC}
Let $f_{\mu} \in\mathfrak{X}_\emph{Byk}^3(\EU^3)$. There is an open region in the space of parameters, say $\tilde{\mathcal{V}}\subset \mathcal{V}$, such that if $\mu\in\tilde{\mathcal{V}}$, then  $\mathcal{F}_\mu$ exhibits:\\
\begin{enumerate}
\item strange attractors with an ergodic SRB measure occuring in a subset of $\tilde{\mathcal{V}}$ with positive Lebesgue measure;\\
\item an open set of initial conditions exhibiting historic behaviour occurring in a subset of $\tilde{\mathcal{V}}$ whose topological closure is $\tilde{\mathcal{V}}$.\footnote{The open set is defined in the phase space; the set $\tilde{\mathcal{V}}$ is defined in the space of parameters.} \\
\end{enumerate}
 \end{maincorollary}

  Taking advantage of the existence of the  Hopf bifurcation surfaces, we may use the reasoning of Denjoy \cite{Denjoy} to conclude the existence of a map $H$, arbitrarily $C^1$-close to $\mathcal{F}_{\mu}$,  with a contracting non-trivial wandering domain.  As defined in Appendix \ref{ss: wandering}, a wandering domain for $\mathcal{F}_{\mu}$ may be seen as  a non-empty connected open set whose forward orbit is a sequence of pairwise disjoint open sets.

\begin{maincorollary}
\label{main_thB}
Let $f_{\mu} \in\mathfrak{X}_\emph{Byk}^3(\EU^3)$.  There is an open region in the space of parameters, say $\tilde{\mathcal{V}}\subset \mathcal{V}$, such that if $\mu\in\tilde{\mathcal{V}}$, then:\\
 
\begin{enumerate}
\item  $\mathcal{F}_{\mu}$ exhibits an attracting 2-dimensional torus, which is contractible\footnote{The invariant circles, in the first return map, do not envelop  the phase cylinder.}; \\ 
 
\item there exists a diffeomorphism $H$ arbitrarily $C^1$-close to $\mathcal{F}_{\mu}$, exhibiting a contracting non-trivial wandering domain $D$ for which the union of the $\omega$-limit set of points in $D$  is a nonhyperbolic transitive Cantor set without periodic points.\\

\end{enumerate}

 \end{maincorollary}

The proofs of Corollaries \ref{Corol1}, \ref{main_thC} and \ref{main_thB} are performed in Sections \ref{proof_corol1}, \ref{proof_corol3} and \ref{proof_corol2}, respectively. 
In Section \ref{dissection}, we analyse the continuation of these bifurcations by studying the precise expressions of the eigenvalues of $D\mathcal{F}_\mu$ at the $(1, \ell)$-fixed points of $\mathcal{F}_\mu$. We derive an analytical expression for the Hopf and period-doubling bifurcations, as well as for the transitions \emph{node $\leftrightarrow$ focus}. 
These plethora of bifurcations (among others),  limited by two saddle-node bifurcation surfaces,  is what we call a \emph{resonance wedge}. An \emph{Arnold tongue} may be seen as the projection of one of these wedges.

\begin{remark}
\label{rem: weakly}
The Bogdanov-Takens bifurcation will be computed under the assumption that $M= A \pm \lambda$ (see Prop. \ref {propBT}). Since $A$ and $\lambda$ are small, then $M$ must be also small.  This condition is achieved by assuming Hypothesis  \textbf{(P5)}: if $\delta \gtrsim 1$, then $M$ is small (cf.  Figure~\ref{graph_M}).   
\end{remark}

 \begin{figure}[h]
\begin{center}
\includegraphics[height=4cm]{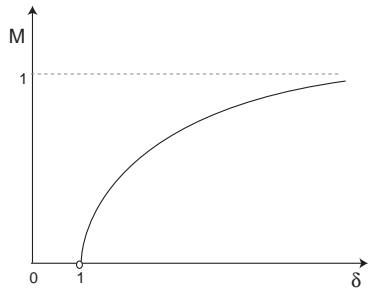}
\end{center}
\caption{\small  Graph of $M(\delta)=   \dpt \delta ^{\frac{1}{1-\delta}}- \delta ^{\frac{\delta}{1-\delta}}$ for $\delta>1$.} 
\label{graph_M}
\end{figure}

\section{The first return map}\label{localdyn}

We analyze the dynamics near  $\Gamma$ through local maps, after selecting appropriate coordinates in neighborhoods of   $O_1$ and $O_2$.

\subsection{Local coordinates}
In order to describe the dynamics around the cycles of $\Gamma$, we use the local coordinates near the equilibria $O_1$ and $O_2$ introduced in  \cite{OS} (cf. \cite{LR2016, Rodrigues2019}).
In these coordinates, we  use cylindrical neighborhoods  $V_1$ and $V_2$  in ${\RR}^3$ of $O_1 $ and $O_2$, respectively, of radius $\rho=\tilde\varepsilon>0$ and height $z=2\tilde\varepsilon$.
After a linear rescaling, we  assume   $\tilde\varepsilon=1$.

\medbreak
The boundaries of $V_1$ and $V_2$ consist of three components: the cylinder wall parametrised by $x\in \RR\pmod{2\pi}$ and $|y|\leq 1$ with the cover $$ (x,y)\mapsto (1 ,x,y)=(\rho ,\theta ,z)$$ and two discs, the top and bottom of the cylinder. We consider polar coverings of these disks $$(r,\phi )\mapsto (r,\phi , \pm 1)=(\rho ,\theta ,z)$$
where $0\leq r\leq 1$ and $\varphi \in \RR\pmod{2\pi}$.
In $V_1$, we use the notation: \\
\begin{itemize}
\item
$\In(O_1)$, the cylinder wall of $V_1$,  consists of points that go inside $V_1$ in positive time; \\
\item
$\Out(O_1)$, the top and bottom of $V_1$,  consists of points that go outside $V_1$ in positive time. \\
\end{itemize}
We denote by $\In^+(O_1)$ the upper part of the cylinder, parametrised by $(x,y)$, $y\in\, ]\, 0,1]$ and by $\In^-(O_1)$ its lower part. The local stable manifold of $O_1$, $W^s_\loc(O_1)$, corresponds to the circle parametrised by $ y=0$.
 The cross-sections around $O_2$ are dual of the previous sections. The set $W^s_\loc (O_2)$ corresponds to the intersection of the $z$-axis with the top and bottom of  $V_2$; these two intersection points will be the origin of its coordinates. The set 
$W^u_\loc (O_2)$ is defined by $y=0$ and: \\

\begin{itemize}
\item
$\In(O_2)$, the top and bottom of $V_2$,  consists of points that go inside $V_2$ in positive time; \\
\item
$\Out(O_2)$,  the cylinder wall  of $V_2$,  consists of points that go outside  $V_2$ in positive  time, with $\Out^+(O_2)$ denoting its upper part, parametrised by $(x,y)$, $y\in \, ] \,0,1]$ and $\Out^-(O_2)$  its lower part parametrised by $(x,y)$, $x\in \RR$ and $y\in \, [\, -1,0\, [$. \\
\end{itemize}

By construction, the flow is transverse to these cross-sections and the boundaries of $V_1$ and of $V_2$ may be written as the topological closure of  $\In(O_1) \cup \Out (O_1)$ and  $\In(O_2) \cup \Out (O_2)$, respectively. 
\begin{remark}
The orientation of the angular coordinate near $O_2$ is chosen to be \emph{compatible} with the direction induced by the angular coordinate in $O_1$.
\end{remark}

\subsection{Local maps}
Adapting \cite{OS}, the trajectory of  a point $(x,y) \in \In^+(O_1)$, leaves $V_1$ at
 $\Out(O_1)$ at
\begin{equation}
\Phi_{1 }(x,y)=\left(y^{\delta_1} + S_1(x,y; A, \lambda, \omega),x-\frac{\omega \, \ln y}{ E_1}+S_2(x,y; A, \lambda, \omega) \right)=(r,\phi)
\label{local_v}
\end{equation}
 where $\dpt \delta_1=\frac{C_{1 }}{E_{1}} > 1$, 
$S_1$, $S_2$ are smooth functions which depend on the parameters $A$, $\lambda$ and $\omega$ and satisfy:
\begin{equation}
\label{diff_res}
\left| \frac{\partial^{k+l+m}}{\partial x^k \partial y^l  \partial A ^{m_1}   \partial \lambda ^{m_2}   \partial \omega ^{m_3}   } S_i(x, y;A, \lambda, \omega)
\right| \leq \, C\, \,  y^{\delta_1 + \sigma - l},
\end{equation}
where the numbers $C$, $\sigma$ are positive constants and $k, l, m_1, m_2, m_3$ are non-negative integers. In a similar way, a point $(r,\phi)$ in $\In(O_2) \backslash W^s_{\loc}(O_2)$ leaves $V_2$ at $\Out(O_2)$ at
\begin{equation}
\Phi_{2 }(r,\phi )=\left(\phi -\frac{\omega\, \ln r}{E_2} + R_1(r,\phi ; A, \lambda, \omega),r^{\delta_2 }+R_2(r,\phi; A, \lambda, \omega )\right)=(x,y)
\label{local_w}
\end{equation}
where $\dpt \delta_2=\frac{C_{2 }}{E_{2}} >1 
$ and $R_1$, $R_2$ satisfy a  condition similar  to (\ref{diff_res}). The expressions $S_1$, $S_2$,  $R_1$, $R_2$ correspond to terms that vanish when $y$ and $r$ go to zero.   
\bigbreak

\subsection{Global maps}\label{transitions}
The coordinates on $V_1$ and $V_2$ are chosen so that $[O_1\rightarrow O_2]$ connects points with $z>0$ (resp. $z<0$) in $V_1$ to points with $z>0$  (resp. $z<0$) in $V_2$. Points in $\Out(O_1) \setminus W^u_{\loc}(O_1)$ near $W^u(O_1)$ are mapped into $\In(O_2)$ along a flow-box around each of the connections of $[O_1\rightarrow O_2]$. Assuming \textbf{(P9)},  the transition
$$\Psi_{1 \rightarrow  2}\colon \quad \Out(O_1) \quad \rightarrow  \quad \In(O_2)$$
does not depend neither on $\lambda, A$ nor $\omega$ and is the \emph{Identity map}, a choice compatible with  \textbf{(P5)} and \textbf{(P7)}. 
 Denote by $\eta$ the  map:
$$\eta=\Phi_{2} \circ \Psi_{1 \rightarrow  2} \circ \Phi_{1 }\colon \quad \In(O_1) \backslash W^s_\loc (O_1) \quad \rightarrow  \quad \Out(O_2).$$
Omitting the higher order terms that appear in  \eqref{local_v} and \eqref{local_w}, for $y>0$ we may write:
\begin{equation}\label{eqeta}
\eta(x,y)=\left(x-K \omega \ln  y \,\,\,\pmod{2\pi}\, \, , \,y^{\delta} \right)
\end{equation}
with
\begin{equation}\label{delta e K}
\delta=\delta_1 \delta_2 \gtrsim1 \qquad \text{and} \qquad  K= \frac{C_1 +E_2 }{E_1 E_2} > 0.
\end{equation}

A similar expression is valid for $y<0$, after suitable changes. 
Using  \textbf{(P8)} and \textbf{(P9)}, for all $A> \lambda\geq 0$ and $\omega \in \RR^+$, we may define the  map
$\Psi_{2 \rightarrow  1}:\Out(O_2)\rightarrow  \In(O_1)$
that depends on the parameters $\lambda$ and $A$ (see Figure \ref{transitions}):
\begin{equation}\label{transition21}
\Psi_{2 \rightarrow  1}(x,y)=\left(x, \,y +A+\lambda \Phi(x) \right) \qquad \text{where}\quad \Phi(x)=\sin x.
\end{equation}
Observe that $\Psi_{2 \rightarrow  1}$ does not depend on $\omega$.  
\begin{figure}[h]
\begin{center}
\includegraphics[height=5.0cm]{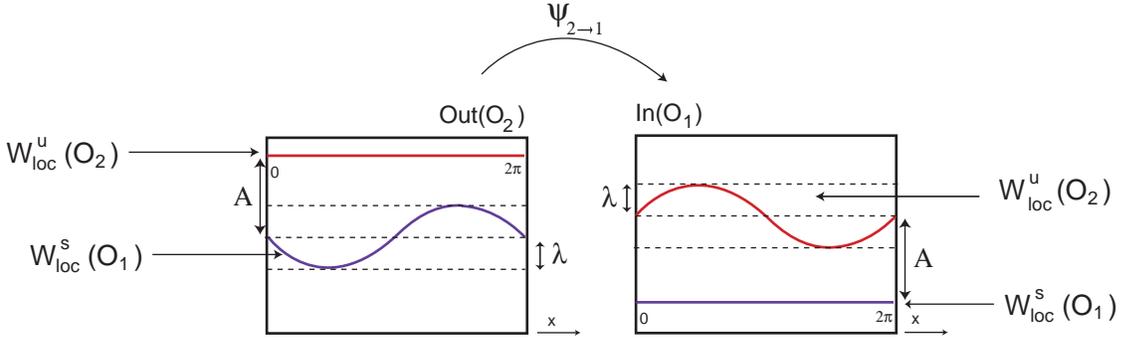}
\end{center}
\caption{\small Geometry of the global map $\Psi_{2 \rightarrow  1}$. Scheme inspired by Gaspard \cite{Gaspard}.} 
\label{transitions}
\end{figure}
The expression of the first return map $\mathcal{F}_{\mu}$ follows by composing the local and global maps constructed above.  Let
\begin{equation}\label{first return 1}
\mathcal{F}_{\mu} =  \eta \circ \Psi_{2 \rightarrow  1}=  \quad  \mathcal{D} \subset \Out(O_2) \quad \rightarrow  \quad \mathcal{D} \subset \Out(O_2)
\end{equation}
be the first return map to $\Out(O_2)$, where $\mathcal{D}\neq \emptyset$ is the set of initial conditions $(x,y) \in \Out(O_2)$ whose solution returns to $\Out(O_2)$. Composing $\eta$ (\ref{eqeta}) with  $\Psi_{2 \rightarrow  1}$ (\ref{transition21}), the  expression of $\mathcal{F}_{\mu}$ is given by
\label{first1}
\begin{eqnarray*}
\mathcal{F}_{\mu}(x,y)&=& \left[ x-K\, \omega \ln  \left[y+A+{\lambda}\sin x\right] \,   \pmod{2\pi}, \, \, \left(y + A+{\lambda}\sin x\right)^\delta\right]\\
&=:& \left(\mathcal{F}_1^{\mu}(x,y), \mathcal{F}_2^{\mu}(x,y)\right).
\end{eqnarray*}
 The following remarks will be useful in the sequel.

\begin{remark}
The map $\mathcal{F}_{\mu}$ is $C^3$ and is well defined in a compact subset of $\Out(O_2)$. Thus, results on \emph{circloid maps} \cite{PPS} may be applied to $\mathcal{F}_{\mu}$.
\end{remark}

\begin{remark}
\label{rig rotation}
If $A=\lambda=0$ and $\omega \in \RR^+$, then $\mathcal{F}_1^{(0,0,\omega)}(x,y)=x-K\, \omega \ln y $ may be identified with a \emph{rigid rotation} on $\EU^1=\RR/2\pi$ and $\mathcal{F}_2^{(0,0,\omega)}(x,y)= y^\delta $ defines an invariant contracting foliation. 
\end{remark}

\begin{remark}
\label{rem_dissipative}
 Since $\delta>1$  and $1>\varepsilon>A>y\geq 0$, we may write:
\begin{eqnarray*}
\left|\frac{\partial  \mathcal{F}_2^{\mu} (x,y)}{\partial y } \right|&=& \left|\delta (y + A + \lambda\sin x)^{\delta-1} \right|= \mathcal{O}((A+\lambda)^{\delta-1})<1,
\end{eqnarray*}
where $\mathcal{O}\left((A+\lambda)^{\delta-1}\right)$ represents the standard \emph{Landau notation}.
This means that, under Hypotheses \textbf{(P1)--(P8b)--(P9)}, if  $A>0$ small enough, then  $\mathcal{F}_2^{\mu}$ is a contraction in the variable $y$.
\end{remark}

\section{Proof of Theorem \ref{thm:0} and its corollaries}
\label{proof_th_A}
The main goal of this section is to prove Theorem \ref{thm:0} and its consequences. We start this task by giving preparatory results. 

\subsection{Fixed points of $\mathcal{F}_{\mu}$ and their stability}
The $(1,\ell)$--fixed points of $\mathcal{F}_{\mu}$ in $\mathcal{D}\cap \Out(O_2)$, say $p_\ell=(x_\ell, y_\ell)\in \Out(O_2)$, $\ell \in \NN$, are solutions of
:\begin{equation*}
\left\{
\begin{array}{l}
x-K \, \omega  \log (y+A+\lambda \sin x)=x+2\ell \pi\\ \\
(y+A+\lambda \sin x)^\delta =y.
\end{array}
\right.
\end{equation*}
Therefore,
\begin{equation}
\label{primeira}
y_\ell+A+\lambda \sin x_\ell =\exp \left(\frac{-2\ell \pi}{K\, \omega}\right)\qquad \text{and} \qquad y_\ell =  \exp \left(\frac{-2\ell \pi \delta }{K \, \omega}\right) ,
\end{equation}
which implies that
\begin{equation}\label{lambda maximo}
A+\lambda \sin x_\ell = \exp \left(\frac{-2\ell \pi}{K \, \omega}\right) - \exp \left(\frac{-2\ell \delta \pi}{K \, \omega}\right).
\end{equation}
For $\ell \in \NN$, define the real-valued map
\begin{equation}
\label{def_G}
G_\ell(\omega)= \exp \left(\frac{-2\ell \pi}{K\,  \omega}\right) - \exp \left(\frac{-2\ell \delta \pi}{K\, \omega}\right), \qquad \omega\in \RR^+.
\end{equation}
whose graph is depicted in Figure \ref{graph_G1}, for different values of $\ell \in \NN$. Observe that:
\begin{equation}
\label{fixed_points}
A+\lambda \sin x_\ell =G_\ell ( \omega).
\end{equation}

\begin{figure}[h]
\begin{center}
\includegraphics[height=5.5cm]{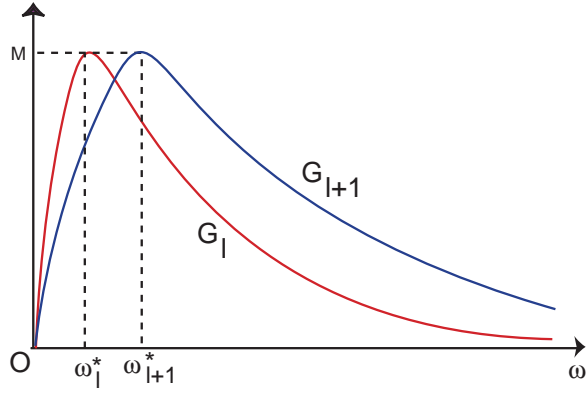}
\end{center}
\caption{\small The map $G_\ell$ has a global maximum $\dpt M= \delta ^{\frac{1}{1-\delta}}- \delta ^{\frac{\delta}{1-\delta}}$ at $\dpt \omega^\star_\ell= \frac{2\ell \pi (\delta-1)}{\ln \delta}>0$.} 
\label{graph_G1}
\end{figure}

The next result summarises some basic properties of $G_\ell$.
\begin{lemma}
\label{lemaF}
For  $\ell, \ell_1, \ell_2 \in \NN$, the following assertions are true:

\begin{enumerate}
\item The map $G_\ell$ has a global maximum $\dpt M= \delta ^{\frac{1}{1-\delta}}- \delta ^{\frac{\delta}{1-\delta}}$ at $\dpt \omega^\star_\ell= \frac{2\ell \pi (\delta-1)}{\ln \delta}>0$. \\
\item The maximum of $G_\ell$ is independent of $\ell$. \\
\item $\dpt \lim_{\omega\rightarrow 0^+}G_\ell(\omega)=\lim_{\omega\rightarrow +\infty}G_\ell (\omega)=0$. \\
\item if $\ell_1<\ell_2$, then $ \omega^\star_{\ell_1}<\omega^\star_{\ell_2}$. \\
\end{enumerate}

\end{lemma}

\begin{proof} Differentiating $G_\ell$ with respect to $\omega$ and multiplying by $K>0$, we get:
\begin{eqnarray*}
K \, G_\ell'(\omega)&=&  \frac{2\ell \pi}{\omega^2}\exp \left(\frac{-2\ell \pi}{\omega}\right)-  \frac{2\ell \delta \pi}{\omega^2}\exp \left(\frac{-2\ell \delta \pi}{\omega}\right) \\ \\
&=&  \frac{2\ell \pi}{\omega^2} \left[ \exp \left(\frac{-2\ell \pi}{\omega}\right)-  \delta \exp \left(\frac{-2\ell \delta \pi}{\omega}\right)\right] \\ \\
&=&  \frac{2\ell \pi}{\omega^2} \exp \left(\frac{-2\ell \pi}{\omega}\right) \left[ 1-  \delta \exp \left(\frac{-2\ell (\delta-1) \pi}{\omega}\right)\right].
\end{eqnarray*}
Since $\dpt \frac{2\ell \pi}{\omega^2} \exp \left(\frac{-2\ell \pi}{\omega}\right)>0$ for all $\omega\in \RR^+$ and $\ell \in \NN$, we may conclude that:

\begin{eqnarray*}
G_\ell'(\omega)=0&\Leftrightarrow&  1-  \delta \exp \left(\frac{-2\ell (\delta-1) \pi}{\omega}\right)=0 \\ \\
&\Leftrightarrow&  \exp \left(\frac{-2\ell (\delta-1) \pi}{\omega}\right)=1/\delta \\ \\
&\Leftrightarrow&   \frac{2\ell (\delta-1) \pi}{\omega} =\ln \delta  \\ \\
&\Leftrightarrow&  \omega= \frac{2\ell \pi (\delta-1)}{\ln \delta}=: \omega^\star_\ell. \\
\end{eqnarray*}
As suggested  in Figure \ref{graph_G1}, it is easy to check that $G_\ell'(\omega)>0$ if $\omega \in \, \left ]\, 0,  \omega^\star_\ell \,  \right [$ and $G'_\ell(\omega)<0$ otherwise. This implies that $G_\ell$ is increasing in $\left ]\, 0,  \omega^\star_\ell \,  \right [$ and decreasing in $]\, \omega^\star_\ell, \infty[$. Furthermore,\\
\begin{eqnarray*}
G_\ell(\omega^\star_\ell)&=& \exp \left(\frac{-2\ell \pi \ln \delta }{2 \ell \pi (\delta-1) }\right) - \exp \left(\frac{-2\ell \delta \pi \ln \delta }{2 \ell (\delta-1) \pi }\right)  \\ \\
&=& \exp \left(\frac{ \ln \delta }{1-\delta }\right) - \exp \left(\frac{ \delta  \ln \delta }{ 1-\delta }\right) \\ \\
&=&\dpt  \delta ^{\frac{1}{1-\delta}}- \delta ^{\frac{\delta}{1-\delta}}=:M. 
\end{eqnarray*}

  The following two limits are zero as a result of the analytic expression of $G_\ell$:
  \begin{eqnarray*}
\lim_{\omega \rightarrow 0^+} \left[ \exp\left(\frac{2\ell \pi}{K\, \omega}\right)  -\exp\left(\frac{2\ell \delta \pi}{K\, \omega}\right)                       \right] = 0= \lim_{\omega \rightarrow +\infty} \left[ \exp\left(\frac{2\ell \pi}{K\, \omega}\right)  -\exp\left(\frac{2\ell \delta \pi}{K\, \omega}\right)                       \right]. 
\end{eqnarray*}
\bigbreak
The last assertion follows straightforwardly from the expression of $\omega_\ell^\star$.
\end{proof}

\bigbreak

In order to determine the Lyapunov stability of the fixed points of $\mathcal{F}_\mu$, we compute the derivative of $\mathcal{F}_{\mu}$, at a general $(1, \ell)$--fixed point $(x_\ell,y_\ell)$.

\begin{equation}
\label{matrix1}
D \mathcal{F}_{\mu} (x_\ell,y_\ell)= 
\left(
\begin{array}{lr} \dpt 1 -  \frac{K \, \omega \lambda \cos x_\ell}{y_\ell+A + \lambda \sin x_\ell }&\dpt-\frac{K \, \omega}{y_\ell+A+  \lambda \sin x_\ell}\\
&\\
\lambda \delta (y_\ell+A+ \lambda \sin x_\ell)^{\delta-1}  \cos x_\ell\qquad  &\delta (y_\ell+A + \lambda \sin x_\ell)^{\delta-1} 
 \end{array} \right).
 \end{equation}
\bigbreak
\bigbreak

In order to find the  $(1, \ell)$--fixed points of $\mathcal{F}_{\mu}$, we need to solve the equation: 
\begin{equation}
\label{BB21}
\varphi(x)=G_\ell ( \omega), \qquad \text{where} \qquad  \varphi(x)=A+\lambda \sin x.
\end{equation}

For $A>\lambda>0$ fixed and $\omega \in \RR^+$, the graph of the left hand side 
of \eqref{fixed_points}, say $\varphi(x) $, depends on $x \in [0,2\pi]$   and  does not depend on $\omega$. On the other hand, the graph of the right hand side 
of \eqref{fixed_points}  does  not depend on $x$.

\begin{figure}[h]
\begin{center}
\includegraphics[height=4.0cm]{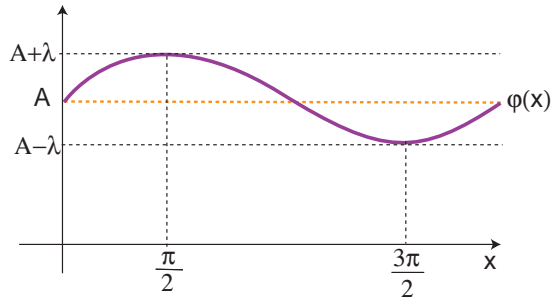}
\end{center}
\caption{\small Graph of $\varphi(x)=A+\lambda \sin x$, $x\in [0, 2\pi]$.}
\label{graph_phi1}
\end{figure}

As illustrated in Figures \ref{graph_phi1} and  \ref{graph_phi2}, finding $(1, \ell)$-fixed points of $\mathcal{F}_\mu$ amounts to intersect the graph of $\varphi(x)$ with a horizontal line.
The line moves first up and then down, as $\omega$ increases.
Since the range of $\varphi$ is the interval $\left[A -\lambda, A+\lambda\right]$, and the range of $G_\ell( \omega)$ is the interval $\left(0,M\right]$, the geometry of the solution set depends on the relative positions of these intervals. From now on, we use the inequality $M \geq A+\lambda$, $A \in [0, \varepsilon]$ -- see \eqref{paramater_set}. The analytic treatment of the other cases are similar to the approach of Section 5 of \cite{LR2020}.

As $\omega$ increases from $0$, there is a threshold value $ \omega_1$ for which the horizontal line at height $G_\ell(\omega_1)$ touches the graph of $\varphi$ at $x=3\pi/2$. At this point we have $\sin(x)=-1$. As $\omega$ increases further, each tangency unfolds as two intersection points of the graph with the horizontal line. There is a saddle-node at the points $\left(x^{(1)}, G_\ell( \omega_1)\right)= \left(3\pi/2,\, G_\ell(\omega_1)\right)$, as we will see in Proposition \ref{propBT}. The surface $G_\ell(\omega)=A\pm \lambda$ defines the boundaries of the $(1,\ell)$--\emph{resonance wedge}. 
The horizontal line may move further up and a pair of  solutions come together at a second saddle-node at
$
\left(x^{(2)}, G_\ell( \omega_2)\right))= \left({\pi}/{2},G_\ell( \omega_2)\right)
$
and reappear at a saddle-node at
$
\left(x^{(3)}, G_\ell( \omega_3)\right)= \left({\pi}/{2},G_\ell( \omega_3)\right)
$
coming together finally at
$
\left(x^{(4)}, G_\ell( \omega_4)\right)= \left({3\pi}/{2},G_\ell( \omega_4)\right).
$
The evolution of the geometry of solutions of \eqref{fixed_points}, as $\omega$ varies, is illustrated on the right side of Figure \ref{graph_phi2}.
We show below that, at these points, the map $D\mathcal{F}_{\mu}$ has an eigenvalue equal to 1.

\bigbreak
Based on \eqref{primeira}, for each $\ell \in \NN$, define the map 
\begin{equation}
\label{def_y}
y(\omega)= \exp \left(\frac{-2\ell \pi \delta }{K\, \omega}\right) ,\qquad  \text{with}\qquad \omega \in \RR^+.
\end{equation}

\subsection{Double eigenvalue 1}
A discrete-time Bogdanov-Takens  bifurcation occurs when the maximum of $G_\ell$ coincides  with either  the minimum or the maximum of $\varphi$. In this subsection, we check the necessary linear conditions for this  bifurcation.  In this section, we implicitly use the fact that the network is weakly attracting (see Remark \ref{rem: weakly}).

\begin{proposition}\label{propBT}
For $G_\ell(\omega_\ell^\star)=A\pm \lambda$, the derivative $D \mathcal{F}_{\mu}\left(x^{(N)},y(\omega_\ell^\star)\right)$ at a solution of  \eqref{fixed_points}  has 1 as  a double eigenvalue and is not the identity, for $N=1,..., 4$.
\end{proposition}

\begin{figure}[h]
\begin{center}
\includegraphics[height=4.0cm]{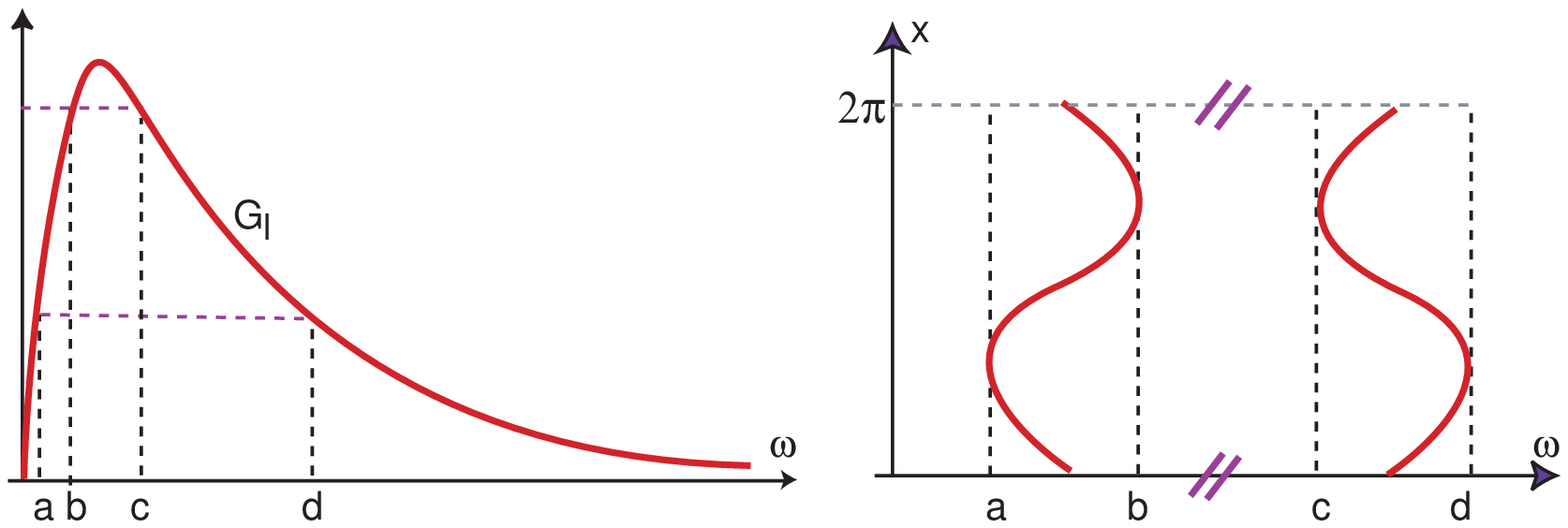}
\end{center}
\caption{\small Graph of $G_\ell$ and evolution of the fixed points when $\omega$ varies. $a=\omega_1$, $b=\omega_2$, $c=\omega_3$ and $d=\omega_4$. Double bars mean that the sides are identified.  }
\label{graph_phi2}
\end{figure}
\begin{proof}
Computing the derivative
$D \mathcal{F}_{\mu}$ at the  points $\left(x^{(N)},y(\omega_N)\right)$, $N=1,\ldots,4$, where $\sin(x_N)=\pm 1$, we  get:
$$
 D \mathcal{F}_{\mu} \left(x^{(N)},y( \omega_N)\right)=
 \left(
\begin{array}{lr} 1&\dpt-\frac{K\omega}{y( \omega_N)+A\pm\lambda }\\
&\\
0\qquad &\delta \left(y( \omega_N)+A\pm \lambda\right) ^{\delta-1} 
 \end{array} \right)
$$
At $\left(x^{(N)},y(\omega_N)\right)$ the Jacobian matrix is triangular and so the two eigenvalues are 
$$\Delta_1=1 \qquad \text{and} \qquad \Delta_2=\delta (y(\omega_N)+A\pm \lambda)^{\delta-1} >0.$$
Since $\omega_\ell^\star$ was defined to be the value of $ \omega$ where the function $G_\ell $ defined in \eqref{def_G} has a global maximum, then $\dfrac{dG_\ell}{dt} (\omega_\ell^\star)=0$.
In particular,

\begin{equation}
\begin{array}{lcl}
\Delta_2&=& \delta (y(\omega_\ell^\star)+A\pm \lambda)^{\delta-1} \\ \\
&=& \dpt  \delta \left(\exp \left( \frac{-2\ell \pi }{\omega_\ell^\star}\right)  \right)^{\delta-1} \\ \\
&=& \dpt \delta \left(\exp \left( \frac{-2\ell \pi (\delta-1) }{\omega_\ell^\star}\right)  \right) \\ \\
&=& \dpt \delta \left(\exp \left( \frac{-2\ell \pi (\delta-1) \ln \delta }{2\ell \pi (\delta-1) } \right)  \right) \\ \\
&=& \dpt \delta \exp (\ln \delta^{-1})= 1.  \\ \\
 \end{array}
\end{equation}

Hence the derivative $D \mathcal{F}_{\mu}$, at the  solutions of \eqref{fixed_points} with $G_\ell(\omega_\ell^\star)=A\pm\lambda$, has a double eigenvalue equal to 1, and is not the identity as we may confirm in \eqref{matrix1}.
\end{proof}

\subsection{Proof of Theorem \ref{thm:0}}

\label{prova_pd}

Proposition \ref{propBT} indicates a (possible) bifurcation  of codimension 2, corresponding to a curve in the 3-dimensional parameter space $\mu=(A, \lambda, \omega)$,
where we expect to find a discrete-time Bogdanov-Takens bifurcation.
This bifurcation occurs at points where 1 is a double eigenvalue,  the derivative is not the identity and  the map $\mathcal{F}_\mu$ satisfies a finite number of non-degeneracy conditions. In this section, we check these nonlinear conditions.
We recall the main ideas of \cite{BRS96, Yagasaki} adapted to our purposes.
\bigbreak

Along the surface defined by $\pm \lambda = A - G_\ell (\omega_\ell^\star)$, the map $\mathcal{F}_{\mu} $  has a fixed point  $p_\ell=(x^{(N)},y(\omega_\ell^\star)), N=1,...,4$ and $\ell \in \NN$, such that  $D\mathcal{F}_{\mu}  (p_\ell)$ has a double unit eigenvalue but is not the identity. 

\medbreak
\begin{table}[htb]
\begin{center}
\begin{tabular}{|c|c|c|}  \hline 
  &&\\
Coefficient & \quad  $\dpt x=\frac{\pi}{2}$ \quad  \qquad  &\qquad   $\dpt x=\frac{3\pi}{2}$ \qquad \quad  \\
  &&\\
\hline \hline

&& \\
$a_{20}(\nu)$ &$\dpt  \frac{C K \omega \lambda}{(A-\lambda)^2}<0$ & $\dpt  -\frac{C K \omega \lambda}{(A+\lambda)^2}>0$ \\
&& \\
 \hline
\hline && \\
$b_{11}(\nu)$ &$\dpt -(A-\lambda)^{\delta-2} \lambda \delta (1-\delta)>0$ & $\dpt (A+\lambda)^{\delta-2} \lambda \delta (1-\delta)<0$ \\
&& \\
 \hline
  \hline
  &&\\  
$b_{20}(\nu)$ &$\dpt -(A-\lambda)^{\delta-2} \lambda \delta (1-\delta)>0$ & $\dpt (A+\lambda)^{\delta-2} \lambda \delta (1-\delta)<0$ \\
&& \\ 
 \hline
\end{tabular}
\end{center}
\bigskip
\caption{\small Leading coefficients $a_{20}(\nu)$, $b_{11}(\nu)$ and $b_{20}(\nu)$ of \eqref{Taylor1} used to check the nonlinear conditions for the discrete-time Bogdanov-Takens. Note that $b_{11}(\nu)= b_{20}(\nu)$, for $\nu= (A, \lambda_0, \omega)$ near $ (G_\ell(\omega_\ell^\star) \mp \lambda_0, \lambda_0, \omega_\ell^\star)$. }
\end{table} 
 For $\lambda=\lambda_0>0$ fixed, by composing the translation $(x_\ell, y_\ell) \mapsto (0,0)$ with the following (local) change of coordinates:
$$
(x,y) \mapsto (x, \, C\, y)\qquad \text{where} \qquad C= \frac{-2K  \ell \pi (\delta-1)}{\, \delta^\frac{\delta}{1-\delta}\, \ln \delta }<0,
$$
for $\nu= (A, \lambda_0, \omega)$ near $ (G_\ell(\omega_\ell^\star) \mp \lambda_0, \lambda_0, \omega_\ell^\star)\in \mathcal{V}$, the map $D \mathcal{F}_{\nu} (0,0) $ has the form:
\bigbreak
\begin{equation}
\label{Taylor1}
D \mathcal{F}_{\nu} (x,y) = \left(\begin{array}{cc}
1 & 1 \\
0 & 1
\end{array}\right) 
\left(\begin{array}{c}
x \\
y
\end{array}\right)+
 \left(\begin{array}{c}
a(x,y; \nu) \\
b(x,y; \nu)  
\end{array}\right) +
\mathcal{O}(\| (x,y)\|^3), 
\end{equation}
\bigbreak
\noindent
where the polynomial expansions of order $2$ of $a$ and $b$ may be written as:
$$
a(x,y; \nu) = a_{00}(\nu) + a_{10}(\nu) x+a_{01}(\nu) y+\frac{1}{2}  a_{20}(\nu) x^2+ a_{11}(\nu) xy+\frac{1}{2}  a_{02}(\nu) y^2
$$
and
$$
b(x,y; \nu) = b_{00}(\nu) + b_{10}(\nu) x+b_{01}(\nu) y+\frac{1}{2}  b_{20}(\nu) x^2+ b_{11}(\nu) xy+\frac{1}{2}  b_{02}(\nu) y^2
$$
with
$$
a_{00}(0,0)= a_{10}(0,0)= a_{01}(0,0)= b_{00}(0,0)= b_{10}(0,0)= b_{01}(0,0)=0.
$$

The leading coefficients of $a$ and $b$ that will be used in the sequel are listed in Table 1. We concentrate our attention on the fixed point associated to $x=3\pi/2$; the other is similar. 
\begin{figure}[h]
\begin{center}
\includegraphics[height=9.0cm]{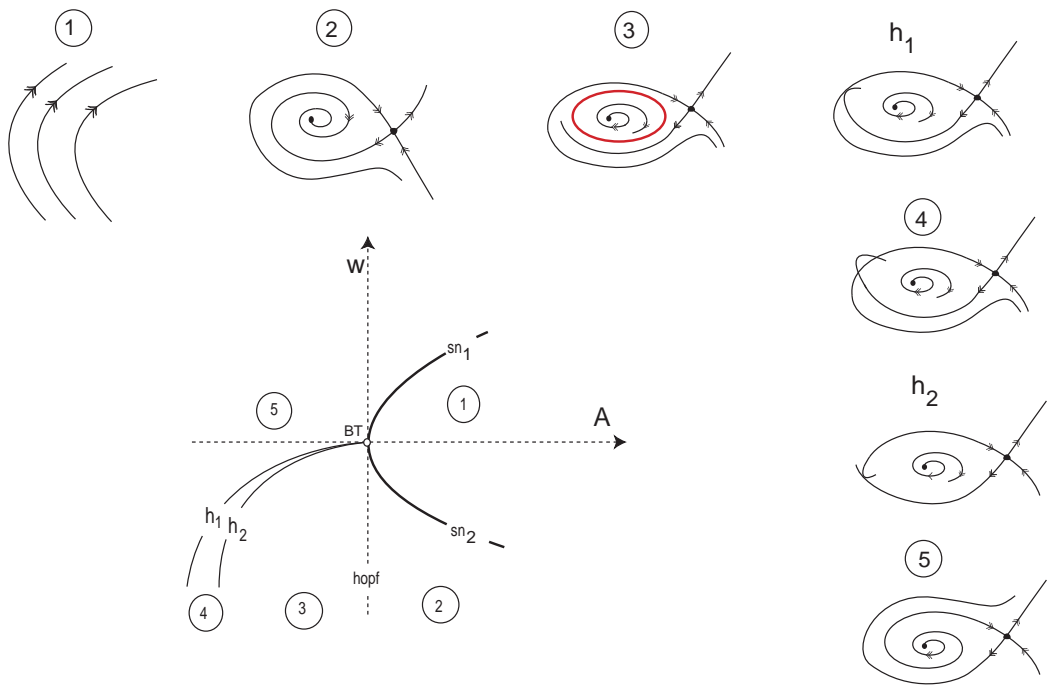}
\end{center}
\caption{\small Dynamics for the discrete-time Bogdanov-Takens bifurcation in  $(A, \omega)$, after a smooth change of coordinates. $\mathbf{sn}_1$ and $\mathbf{sn}_2$: Saddle-node bifurcations; \textbf{hopf}: Hopf bifurcation; $\mathbf{h}_1$ and $\mathbf{h}_2$: homoclinic tangencies associated to a dissipative fixed point. \textbf{1}: no recurrent dynamics; \textbf{2}: the unstable manifold of the saddle is connected with the stable manifold of the focus; \textbf{3}: saddle and a stable periodic orbit; \textbf{4}: horseshoe dynamics; \textbf{5}: the unstable manifold of the focus intersects the stable manifold of the saddle.   }
\label{BTdiscr2}
\end{figure}
By  Proposition 3.1 of Yagasaki \cite{Yagasaki}, since $$b_{20}(\nu)=  (A+\lambda)^{\delta-2} \lambda \delta (1-\delta)<0 \qquad \text{for} \qquad \delta>1,$$ and $$a_{20}(\nu)+b_{11}(\nu)-b_{20}(\nu)= -\frac{C\, K \omega \lambda}{(A+\lambda)^2}=a_{20} > 0,$$ then there exists a bifurcation point of codimension 2 at $(x_N,y(\omega^\star_\ell))$ with $N=1,..., 4$ and $\ell \in \NN$ such that,  nearby  (see Figure \ref{BTdiscr2}): \\
\begin{enumerate}
\item there exist two curves associated to saddle-node bifurcation ($\textbf{sn}_1$ and $\textbf{sn}_2$); \\
\item  there exists one curve associated to a Hopf bifurcation at the stable focus born at the saddle-node bifurcation (\textbf{hopf}); \\
\item there exists a region with a Lyapunov \emph{stable} invariant circle created at the Hopf bifurcation, since 
\begin{equation}
\label{stable circle}
b_{20}(a_{20}+b_{11}-b_{20})<0 
\end{equation}
(all coefficients are computed at $ (G_\ell(\omega_\ell^\star) \mp \lambda_0, \lambda_0, \omega_\ell^\star)$); \\
\item  there exist two curves, $\textbf{h}_1$ and $\textbf{h}_2$, associated to a homoclinic bifurcation where the stable and unstable manifolds of the saddle point born at (1) touch tangencially. The distance between the two homoclinic bifurcation curves is exponentially small with respect to  $\sqrt{\|\nu - ( G_\ell(\omega_\ell^\star) \mp \lambda_0, \omega_\ell^\star)\|)}$;\\
\item the invariant manifolds of a dissipative saddle intersect transversely inside the parameter region between the curves $\textbf{h}_1$ and $\textbf{h}_2$ and do not intersect outside it.
\end{enumerate}

\begin{remark}
\label{tangency1}
In Proposition 3.1 of \cite{Yagasaki}, there exists an extra condition: $\det D_\mu \nu(0)\neq 0$. This inequality serves to describe (in some system of coordinates) the explicit expression for the bifurcating curves, which is used to conclude that the \textbf{hopf} bifurcation curve is tangent to the saddle-node bifurcation curves $\textbf{sn}_1$ and $\textbf{sn}_2$, at the bifurcation point. 
\end{remark}

For $\ell \in \NN$, denote by $\mathbf{BT}_\ell^1$ and $\mathbf{BT}_\ell^2$ the two discrete-time Bogdanov-Takens bifurcation in the bifurcation parameter $(A, \lambda)$ such that $A>\lambda=\lambda_0$ which occur for $\omega= \omega^\star_\ell$:
\begin{equation}
\label{notation BT}
\mathbf{BT}_\ell^1  \,  \mapsto A=G_\ell(\omega_\ell^\star)-\lambda, \qquad \mathbf{BT}_\ell^2 \,   \mapsto A=G_\ell(\omega_\ell^\star)+\lambda.
\end{equation}

\subsection{Proof of Corollary \ref{Corol1}}
\label{proof_corol1}
This is a direct corollary of Theorem \ref{thm:0}. For $\lambda=\lambda_0>0$  and $\ell \in \NN$ fixed, there exist two points of Bogdanov-Takens bifurcation for the  map $\mathcal{F}_\mu$ at $p_\ell$: $\mathbf{BT}_\ell^1$  and $\mathbf{BT}_\ell^2$ (see \eqref{notation BT}). As depicted in Figure \ref{arnold_tongue4}, varying smoothly $\lambda \gtrsim 0$ around each Bogdanov-Takens bifurcation: 
\medbreak
\begin{enumerate}
\renewcommand{\theenumi}{(\alph{enumi})}
\renewcommand{\labelenumi}{{\theenumi}}
\item\label{saddlenodes} there exist two surfaces of saddle-node bifurcations ($\textbf{SN}_\ell^1$ and $\textbf{SN}_\ell^2$); \\
\item\label{Hopf} there exists a surface of Hopf bifurcations ($\textbf{Hopf}_\ell$); \\
\item\label{homoclinic} there exist two surfaces of  homoclinic tangencies ($\textbf{H}_\ell^1$ and $\textbf{H}_\ell^2$). \\
\end{enumerate}

The two surfaces \ref{homoclinic}  correspond to bifurcations at which the stable and unstable manifolds of a dissipative saddle point are tangent.
In the region between these surfaces there is a transverse intersection of the stable and the unstable manifolds of a saddle.  This configuration implies that the dynamics of $\mathcal{F}_\mu$ is equivalent to Smale's horseshoe. 
Last assertion of Corollary \ref{Corol1} follows from Remark \ref{tangency1}.


\begin{remark}
We cannot exclude the possibility that the two surfaces $\textbf{H}^1_\ell$ and  $\textbf{H}^2_\ell$, $\ell \in \NN$,  coincide, although it would be a highly non-generic behaviour.
\end{remark}

\subsection{Proof of Corollary \ref{main_thC}}
\label{proof_corol3}
The existence of $\textbf{H}^1_\ell$ and  $\textbf{H}^2_\ell$, $\ell \in \NN$, shows that there are surfaces in the space of parameters  for which the map $\mathcal{F}_{\mu}$ has a quadratic (generic) homoclinic tangency associated to a dissipative periodic point. Using \cite{MV93}, there exists a positive measure set  $\Delta$ of parameter values, so that for every $\mu\in \Delta\subset \mathcal{V}$, $\mathcal{F}_{ \mu}$ admits a strange attractor  of H\'enon-type with an ergodic SRB measure.  The existence of historic behaviour is a combination of the latter tangencies and  Theorem A of  Kiriki and Soma \cite{KS}.  

\subsection{Proof of Corollary \ref{main_thB}}
\label{proof_corol2}
The first part of the corollary is a straightforward consequence of  the Hopf bifurcation surface of Corollary \ref{Corol1}, from which a stable torus emerge (see \eqref{stable circle}).   
The proof for the second part is a simple inspection of  Theorem B of \cite{Rodrigues2020} (see also references therein).  For the sake of completeness, we list the main steps of the proof:\\
\begin{enumerate}
\item for each $\lambda=\lambda_0>0$, we write explicitly the normal form for the  family of Hopf bifurcation, which creates an attracting invariant circle; \\ 
\item perturb (if necessary) the truncated normal form in order to obtain an irrational rotation on the circle, say $H_1$;\\
\item perturb $H_1$, using Denjoy procedure \cite{Denjoy}, in order to obtain contracting wandering domains. The resulting map is $C^1$-close to $\mathcal{F}_\mu$, $\mu \in \mathcal{V}\cap \textbf{Hopf}_\ell$, $\ell \in \NN$.
\end{enumerate}

\section{Dissecting a resonance wedge:  putting  all together}
\label{dissection}
Theorem \ref{thm:0} may be seen as a ``local'' theorem. A new problem arises: \emph{how the surfaces of bifurcations of Corollary \ref{Corol1} are globally organised?} This section provides a partial answer to this question. We plot the graphs of the maps that defines the Hopf, the \emph{transitions node $\leftrightarrow$ focus} 
and the period-doubling bifurcation, as function of the parameters $\mu=(A, \lambda, \omega)\in \mathcal{V}$. These bifurcations arise in a form consistent with Corollary \ref{Corol1}. 
\subsection{Necessary conditions for bifurcations}
The Hopf surfaces are particularly significant, as they separate the resonance regions into parts with an attracting periodic orbit and
parts with none. 
To simplify the notation, denote by $\textbf{det}$ and $\textbf{trace}$ the determinant and the trace of $D\mathcal{F}_{\mu} (x, y)$ (see \eqref{matrix1}), respectively. \\ 
\begin{eqnarray*}
\mathbf{det}:= \det D\mathcal{F}_{\mu}  (x, y)&=&    \delta (y+A+\lambda \sin x)^{\delta-1} \\ \\
\textbf{trace}:=\text{trace}  D\mathcal{F}_{\mu }  (x, y)&=&  1 -  \frac{K \omega \lambda \cos x}{y+A+\lambda \sin x } + \delta (y+A+\lambda \sin x)^{\delta-1}. \\ 
\end{eqnarray*}
Note that for a $2\times 2$--real matrix  $D\mathcal{F}_{\mu} (x, y)$, its eigenvalues are the roots of the polynomial in $t$ given by $$P(t)=t^2 - \textbf{trace} \, t+\textbf{det}, \qquad \text{say} \qquad t=\frac{\textbf{trace}\pm \sqrt{\textbf{trace}^2-4\textbf{det}}}{2}.
$$

Up to nonlinear conditions, a Hopf bifurcation (for a map) occurs when the norm of the complex (conjugate) eigenvalues crosses the unit circle. This happens when 
$$
  \textbf{det}=1 \qquad \text{and} \qquad  \textbf{trace} \in \,\,   ]-2, 2[. \\
$$
\bigbreak
We know that $G_\ell( \omega)=\varphi(x)$ because $(x,y)$ is a $(1,\ell)$--fixed point of $\mathcal{F}_{\mu} $ -- see \eqref{fixed_points}. This means that 
\begin{eqnarray*}
&&\exp \left(\frac{-2\ell \pi}{K \omega}\right) - \exp \left(\frac{-2\ell \delta \pi}{K \omega}\right) = A + \lambda \sin x  \\
&\Leftrightarrow & G_\ell( \omega) - A = \lambda \sin x. \\
\end{eqnarray*}

Since $\lambda\cos x = \pm \sqrt{\lambda^2 - \lambda^2\sin x} =\pm \sqrt{\lambda^2-(G_\ell(\omega) - A)^2 } $, we may write:

\begin{eqnarray*}
\textbf{trace} &=&  1 - \frac{K \omega \lambda \cos x}{y+A+\lambda \sin x } + \delta (y+A+\lambda \sin x)^{\delta-1}  \\ \\
&=& \dpt1 \mp  \frac{K \omega \sqrt{\lambda^2-(G_\ell( \omega) - A)^2 } }{\exp \left(\frac{-2\ell \pi}{K \omega}\right) } + \delta \exp \left(\frac{-2\ell (\delta -1)\pi}{K \omega}\right) \\  
\textbf{det}   &=&    \delta (y+A+\lambda \sin x)^{\delta-1}  = \delta \, {\exp \left(\frac{-2\ell (\delta-1) \pi}{K \omega}\right) }.
\end{eqnarray*}


\begin{figure}[h]
\begin{center}
\includegraphics[height=4.5cm]{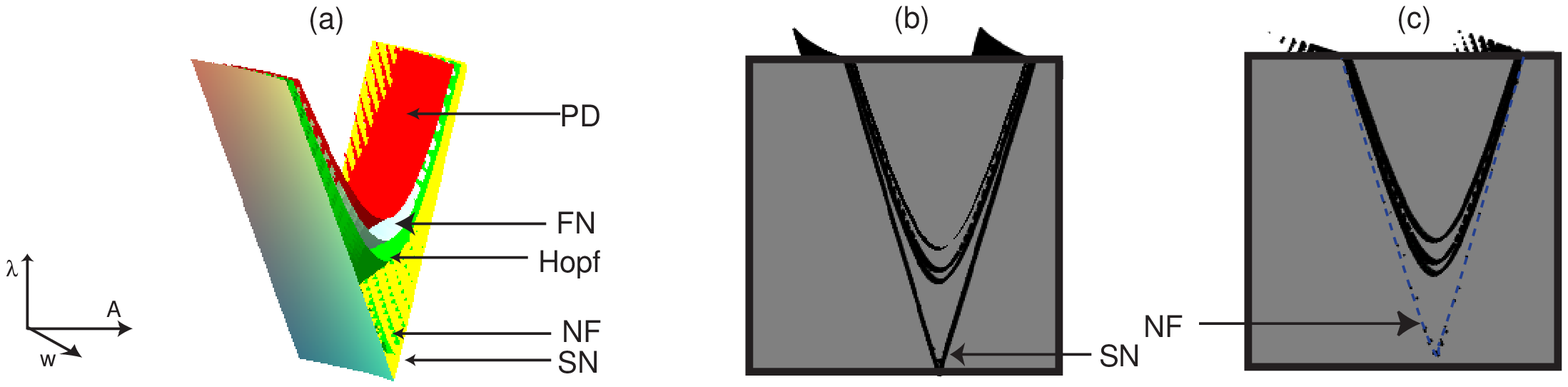}
\end{center}
\caption{\small Graphs of $ \textbf{det}=1$ and $\textbf{trace} \in \, \, ]-2, 2[$ ($\mathbf{Hopf}$), $ {\textbf{trace}^2 -4\,\,  \textbf{det}}=0$ ($\mathbf{NF}$ and $\mathbf{FN}$), $G_\ell (\omega)=A\pm \lambda$ ($\mathbf{SN}$) and one of the roots of $P(t)$ equals to $-1$ ($\mathbf{PD}$), with $\ell=K=1$, $\delta =3$, numerically plotted using \emph{Maple}, $A\in [0; 0.5]$, $\lambda \in [0; 0.1]$, $\omega \in [0.5, 10]$.  In (c)  the $\mathbf{SN}$ surfaces have not been plotted. The $\mathbf{SN}$ and $\mathbf{NF}$ surfaces are almost indistinguishable  in the gray section ($\omega=10$).}
\label{BT11a}
\end{figure}

In Figure \ref{BT11a}, for $\ell=K=2$, $\delta =3$, we have plotted  of the following surfaces, in the 3-parameter space $\mathcal{V}$: \\
\begin{enumerate}
\item \textbf{SN}: the saddle-nodes bifurcations corresponding to $G_\ell (\omega)=A\pm \lambda$; \\
\item \textbf{Hopf}: the Hopf bifurcations corresponding to $ \textbf{trace}\in\, \,  ]-2,2[$ and $\textbf{det} =1$; \\
\item \textbf{NF}/\textbf{FN}: the transitions node $\leftrightarrow$ focus corresponding to $ {\textbf{trace}^2 -4\,\,  \textbf{det}}=0$; \\
\item \textbf{PD}: the period-doubling bifurcation corresponding to the case where one root of $P(t)$ is $-1$. \\

\end{enumerate}

 \begin{figure}[h]
\begin{center}
\includegraphics[height=14.0cm]{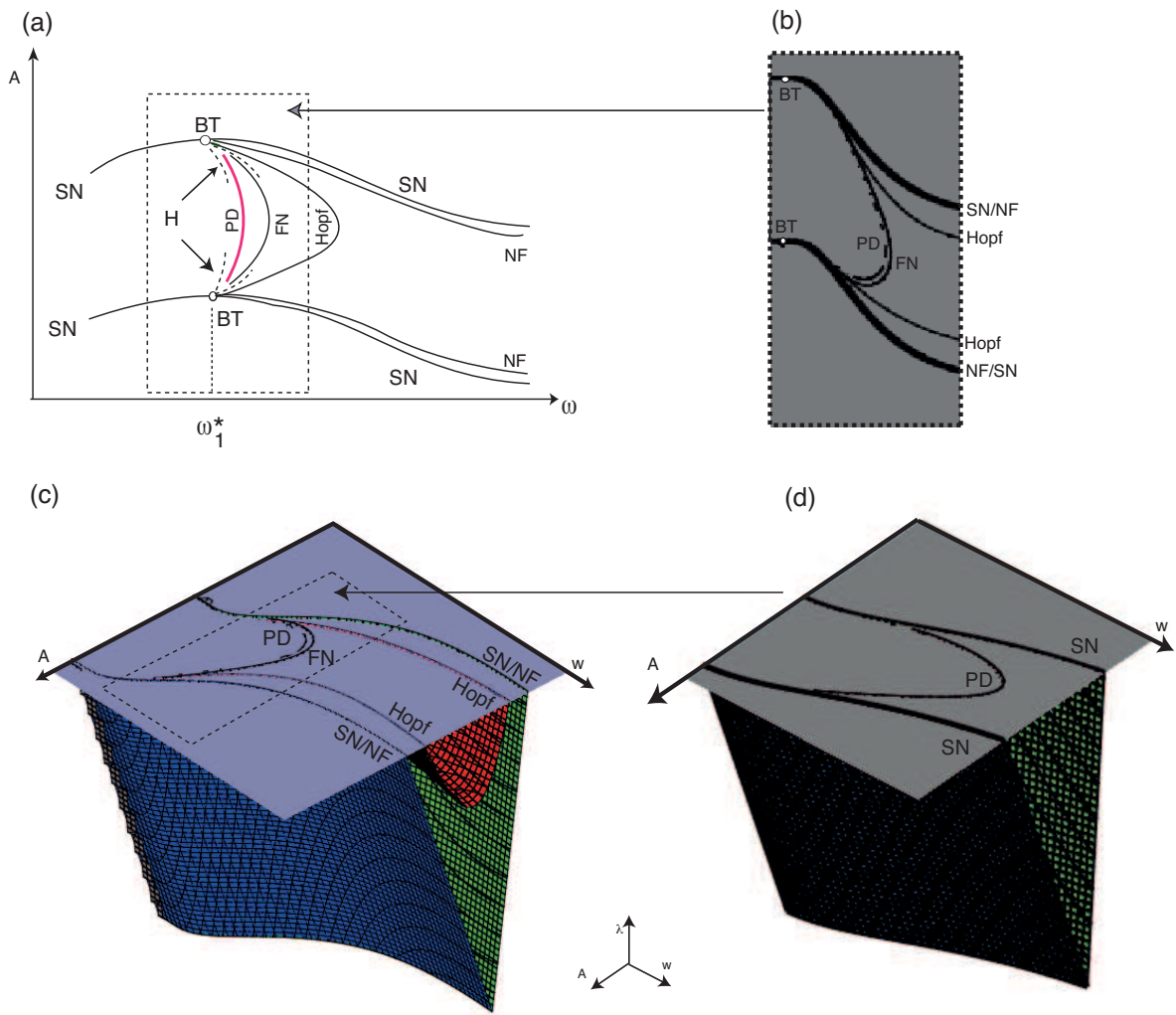}
\end{center}
\caption{\small (a) Plausible  bifurcation diagram in the plane $(\omega, A)$,  for $\lambda=\lambda_0>0$ and $\ell =K=1$. (a): theoretical scheme; (b) numerical scheme using \emph{Maple} with $\ell=K=1$, $\delta =3$, $A\in [0; 0.5]$, $\lambda =0.1$, $\omega \in [2, 10]$. (c) Approximate bifurcation diagram in the space $( A, \lambda, \omega)$,  for $\ell =K=1$, $\delta =3$, $A\in [0; 0.5]$, $\lambda \in [0; 0.1]$, $\omega \in [1, 30]$. (d) Approximate bifurcation diagram in the space $( A, \lambda, \omega)$,  for $\lambda=0.1$ and $\ell =K=1$, $\delta =3$, $A\in [0.1; 0.4]$, $\lambda \in [0; 0.1]$, $\omega \in [2, 10]$.  The gray plane corresponds to $\lambda=0.1$. \textbf{Bifurcations: } $\mathbf{BT} $:  Bogdanov-Takens, $\mathbf{SN}$: saddle-node, $\mathbf{H}$: homoclinic tangencies, \textbf{FN/NF}: transitions focus $\leftrightarrow$ node,  $\mathbf{Hopf}$: Hopf,  $\mathbf{PD}$: period-doubling. In (d), just  the surfaces $\mathbf{SN}$ and $\mathbf{PD}$ have been plotted. }
\label{arnold_tongue4}
\end{figure}

 Numerical plots suggest that the surface $\mathbf{Hopf}_\ell$ connects both ``local'' Hopf surfaces that appear  near  $ \mathbf{BT}_\ell^1$ and $ \mathbf{BT}_\ell^2$ (see Corollary \ref{Corol1} and   Figure \ref{arnold_tongue4}(a)). In the plane defined by $\omega=10$ of Figure \ref{BT11a}, the curves corresponding to $\mathbf{NF}$ and $\mathbf{SN}$ are very close; in (b) of Figure \ref{BT11a}, we cannot distinguish them.  
\begin{figure}[h]
\begin{center}
\includegraphics[height=6.0cm]{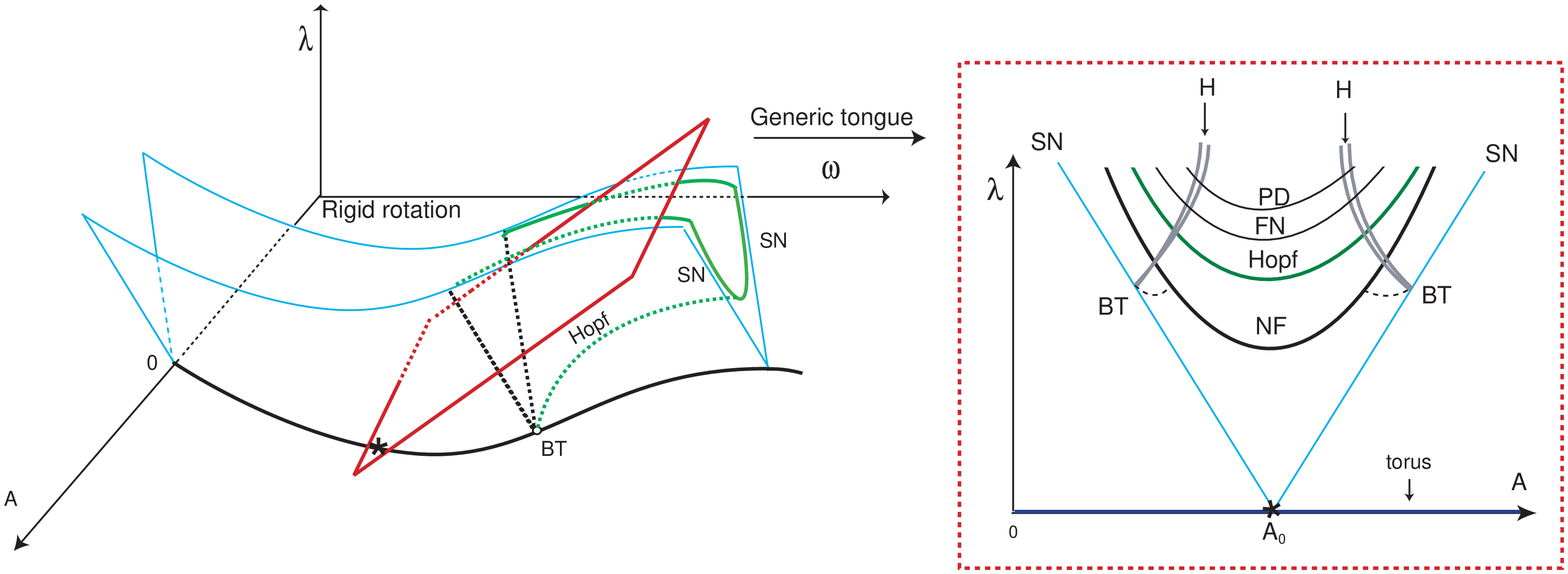} 
\end{center}
\caption{\small Schematic representation of the bifurcations of fixed points of \eqref{fixed_points}, giving rise to  a \emph{resonance wedge} (left) and an  \emph{Arnold tongue} (right) for $A>\lambda\gtrsim 0$. \textbf{Bifurcations: } $\mathbf{BT} $:  Bogdanov-Takens, $\mathbf{SN}$: saddle-node, $\mathbf{H}$: homoclinic tangencies, \textbf{FN/NF}: transitions focus $\leftrightarrow$ node,  $\mathbf{Hopf}$: Hopf,  $\mathbf{PD}$: perio d-doubling.  This figure is distorted and nonlinearly scaled to enable some of the regions to be distinguished.  }
\label{BT11}
\end{figure}

\begin{remark}
On  the $\mathbf{PD}$ surface, one of the multipliers is equal to $-1$ and, in the $\mathbf{BT}$ curve,  both multipliers are equal to $+1$. Therefore, the loci associated to these two bifurcations (for the same fixed point) should never meet.  However, as suggested by  Figure \ref{arnold_tongue4}, restricted to the plane defined by $\lambda=0.1$,   the $\mathbf{PD}$ curve ``terminates'' very close to the $\mathbf{BT}$ bifurcation points. We conjecture that it ``terminates'' within the region limited by  $\textbf{h}_1$ and $\textbf{h}_2$ (homoclinics associated to the discrete $\mathbf{BT}$ bifurcation) due to horseshoe formation/destruction.  
\end{remark}

 \begin{remark}
The theory  of the previous section has been performed for $(1,\ell)$--fixed points of $\mathcal{F}_{\mu}$. However, all these phenomena can also be observed for wedges associated to other rotation number; of course, the analytic expressions for the bifurcation surfaces are different. 
\end{remark}

\begin{figure}[h] 
\begin{center}
\includegraphics[height=12.0cm]{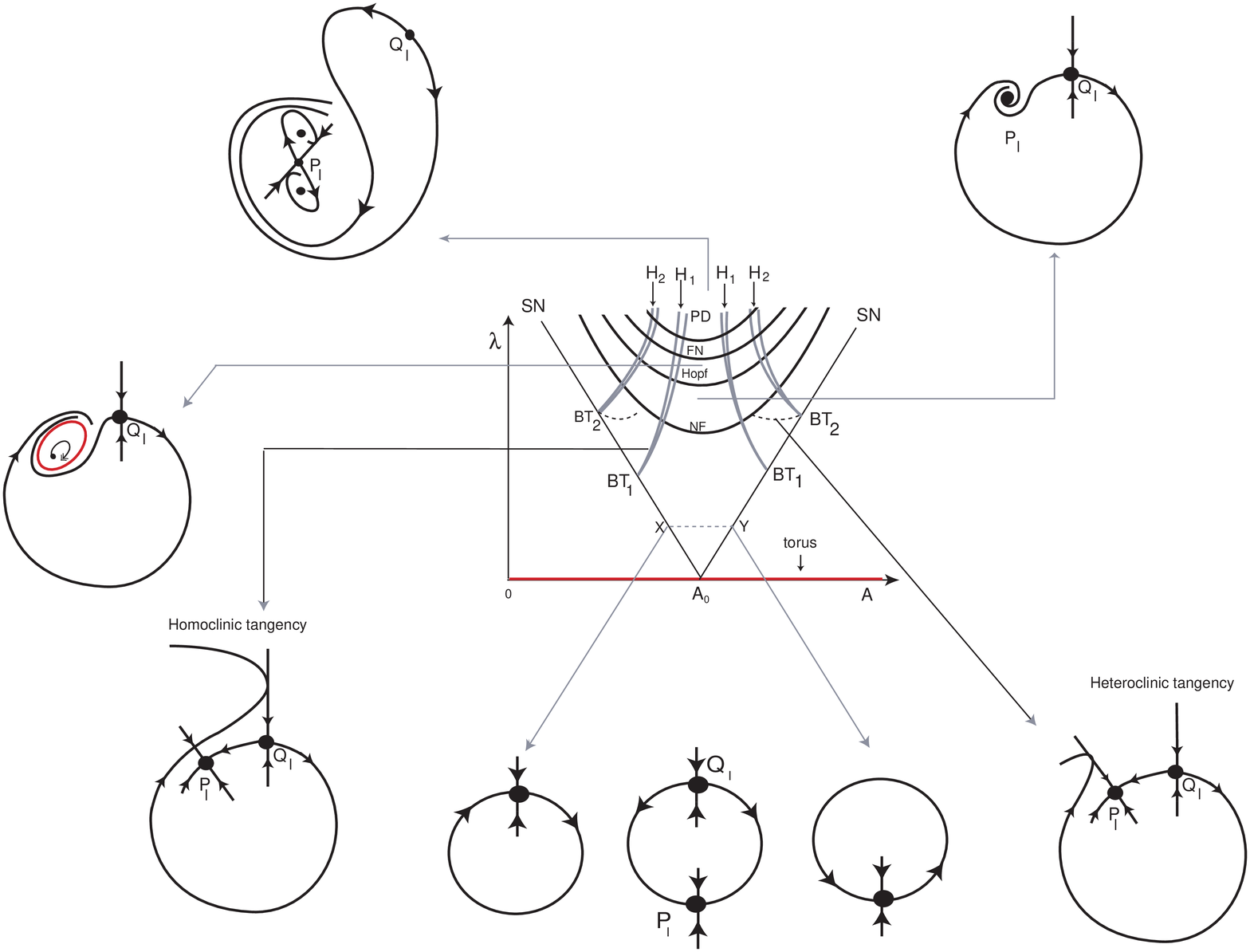} 
\end{center}
\caption{\small  Plausible schematic representation of the  \emph{Arnold tongue} of Figure \ref{BT11}, emphasising the dynamics of $\mathcal{F}_\mu$. \textbf{Bifurcations: } $\mathbf{BT}_1, \mathbf{BT}_2 $:  Bogdanov-Takens, $\mathbf{SN}$: saddle-node, $\mathbf{H}_1, \mathbf{H}_2$: homoclinic tangencies, \textbf{FN/NF}: transitions focus $\leftrightarrow$ node,  $\mathbf{Hopf}$: Hopf,  $\mathbf{PD}$: period-doubling.  The line defined by $\lambda=0$ corresponds to an attracting torus. Along the horizontal path $[XY]$, trajectories change the way of rotation around the torus. }
\label{Arnold_tongue1a}
\end{figure}

\bigbreak

\subsection{Summarizing movie}
\label{plausible}

In this subsection, we give a heuristic discussion of what is going on within an \emph{Arnold wedge}. We compare our results with those found in previous works by other authors;  all results agree well with the theoretical information of  \cite{Ostlund, Shilnikov_tutorial}.

For $\varepsilon>0$ small, the choice of parameters in Section \ref{s:setting}  allows us to build  the bifurcation diagram of Figure \ref{BT11}  in  $\mathcal{V}$ (see \eqref{paramater_set}). 
Dynamical bifurcation surfaces in \emph{resonance wedge} may be projected into a generic plane, giving rise to what the literature calls an \emph{Arnold tongue}. This projection is sketched on the right side of Figure  \ref{BT11} and Figure \ref{Arnold_tongue1a}.   All resonant wedges have the origin as a  common point.  

For $A> \lambda \geq 0$ fixed, if $\omega$ is sufficiently small, the flow of \eqref{general2.1} exhibits a 2-dimensional non-contractible torus which is globally attracting and normally hyperbolic. The dynamics of $\mathcal{F}_{\mu}$ is governed by the dynamics of a circle map. There is a positive measure set $\Delta \subset \mathcal{V}$ so that the rotation number of $\mathcal{F}_{\mu}$ is irrational  if and only if $\mu \in \Delta$ \cite{Herman, Herman79}.

 Within $\mathcal{V}$, for each $\ell\in \NN$, we may define  a resonance wedge, denoted by $\mathcal{T}_\ell$, limited by the surfaces  $\textbf{SN}: \, A=G_\ell(\omega)\pm \lambda$ adjoining the graph of  $$A= G_\ell(\omega), \qquad \lambda=0.$$ Parameters within this wedge correspond to first return maps with at least a pair of fixed points: one of the fixed points is  a saddle (say $Q_\ell$); the other point is a sink (say $P_\ell$).  As suggested in Figure \ref{Arnold_tongue1a}, we suppose the existence of just one pair of fixed points for the following analysis, both with the same \emph{rotation number}.

  The borders of   $\mathcal{T}_\ell$ are the bifurcation surfaces $\textbf{SN}:=\textbf{SN}^1_\ell\cup \textbf{SN}^2_\ell$ at which the fixed points $Q_\ell$ and $P_\ell$ merge to a saddle-node.   These surfaces   might touch the corresponding surfaces of other wedge, meaning that there are parameter values for which  periodic points of periods $m$ and $ \ell$ coexist, $\ell, m\in \NN$.  
   The surface $\textbf{Hopf}$  seems to connect both Hopf surfaces that appear  near  $ \mathbf{BT}_\ell^1$ and $ \mathbf{BT}_\ell^2$ given by Corollary \ref{Corol1}, from where a stable 2-torus emerges. 
   This torus is contractible because it does not envelope the cylinder $\overline{\Out(O_2)}$; the new tori might  coexist (for different values of $\ell \in \NN$) and are not diffeomorphic to the original torus that exists for $\lambda=0$ and $A>0$.

In the bifurcation plane $(A, \lambda)$,  the  \textbf{Hopf} surface is  above the set \textbf{NF} where the eigenvalues of the sink $P_\ell$ become complex. At the period-doubling bifurcation surface \textbf{PD}, one multiplier becomes equal to $-1$.  At this stage, the closed curve resulting from the intersection of the torus with a global cross section,  is no longer homeomorphic to a circle. Furthermore, the torus is no longer smooth as the unstable manifold of the saddle $Q_\ell$ winds around the focus infinitely many times (see Figure \ref{Arnold_tongue1a}). Between the curves \textbf{Hopf} and  \textbf{PD}, the eigenvalues of $P_\ell$ become real again, along a surface \textbf{FN}. These bifurcations have been  discussed in \cite{Ostlund, WY} where the authors relate the dynamics of an Arnold tongue with maps on the circle.     

Along the bifurcation surfaces $\textbf{H}^1_\ell$ and $\textbf{H}^2_\ell$ described by Corollary \ref{Corol1}, one observes a homoclinic contact of  the components $W^s(Q_\ell)$ and $W^u(Q_\ell)$, where $Q_\ell$ is a dissipative saddle for $\mathcal{F}_\mu$. There are small regions (in terms of measure)  inside the resonance wedges where chaotic trajectories are  observable: they correspond to strange attractors of H\'enon type \cite{Rodrigues2019}.
Other stable points of large period exist in the region above the surfaces $\textbf{H}^1_\ell$ and $\textbf{H}^2_\ell$, as a consequence of Newhouse phenomena \cite{Newhouse79}.
Numerics in \cite{BST98} also suggest the existence of \emph{bistability} for open regions of the parameter space: coexistence of a stable periodic solution and an attracting torus.

\section{An example}
\label{s:example}
Our study was initially motivated by the following example introduced in  \cite{Aguiar_tese} and explored in \cite{CastroR2019}. Some preliminaries about symmetries of a vector field may be found in Appendix \ref{app: symmetry}.
 For $\tau_1, \tau_2 \in \,[0,1]$, our object of study is the two-parameter family of vector fields on $\RR^{4}$ $$x=(x_1,x_2,x_3,x_4)\in\RR^4 \quad \mapsto \quad  f_{(\tau_1, \, \tau_2)}(x)$$
defined for each $x=(x_1,x_2,x_3,x_4)\in \RR^4$ by
\begin{equation}\label{example}
\left\{
\begin{array}{l}
\dot{x}_{1}=x_{1}(1-r^2)-\textcolor{red}\omega x_2-\alpha x_1x_4+\beta x_1x_4^2 +\textcolor{blue}{\tau_2}{ x_1x_3x_4} \\ 
\dot{x}_{2}=x_{2}(1-r^2)+\textcolor{red}{\omega}x_1-\alpha x_2x_4 +\beta x_2x_4^2 \\ 
\dot{x}_{3}=x_{3}(1-r^2)+\alpha x_3x_4+\beta x_3x_4^2+\textcolor{magenta}{\tau_1}{ x_4^3} -\textcolor{blue}{\tau_2}{  x_1^2x_4}\\ 
\dot{x}_{4}=x_{4}(1-r^2)-\alpha (x_3^2-x_1^2-x_2^2)-\beta x_4(x_1^2+x_2^2+x_3^2)-\textcolor{magenta}{\tau_1}{ x_3x_4^2}\\
\end{array}
\right.
\end{equation}
where $\dpt \dot{x}_i=\frac{\partial x_i}{\partial t},$  $r^2=x_{1}^{2}+x_{2}^{2}+x_{3}^{2}+x_{4}^{2}$, and

$$
\omega>0, \qquad \beta <0<\alpha, \qquad \beta^2<8 \alpha^2 \qquad \text{and} \qquad |\beta|<|\alpha|.
$$
The vector field $f_{(0,0)}$ is equivariant under the action of the compact Lie group $\mathbb{SO}(2)(\gamma_\psi)\oplus \ZZ_2(\gamma_2)$, where $\mathbb{SO}(2)(\gamma_\psi)$ and $\ZZ_2(\gamma_2)$ act on $\RR^4$ as
$$\gamma_\psi(x_1, x_2,x_3,x_4)=(x_1\cos \psi -x_2 \sin \psi, x_1\sin \psi +x_2\cos \psi, x_3,x_4), \quad \psi \in [0, 2\pi] $$
given by a phase shift $\theta \mapsto \theta+ \psi$ in the first two coordinates, and 
$$ \gamma_2(x_1, x_2,x_3,x_4)=(x_1, x_2,-x_3,x_4).$$
By construction, $\tau_1$ is the controlling parameter of the $\ZZ_2(\gamma_2)-$symmetry breaking and  $\tau_2$ controls the $\mathbb{SO}(2)(\gamma_\psi)-$symmetry breaking but keeping the $\mathbb{SO}(2)(\gamma_\pi)$--symmetry, where
$$
\gamma_\pi (x_1, x_2, x_3, x_4)=(-x_1, -x_2, x_3, x_4). $$
 When restricted to the sphere $\EU^3$, for every $\tau_1, \tau_2 \in [0,1]$, the flow of $f_{(\tau_1, \tau_2)}$ has 
 two equilibria 
$$O_1 =(0,0,0,+1) \quad \quad \text{and} \quad \quad O_2 = (0,0,0,-1), $$
which are hyperbolic saddle-foci. 
The linearization of $f_{(0,0)}$ at $O_1$ and $O_2$ has eigenvalues
$$ -(\alpha-\beta) \pm \omega i, \,\,  \alpha+\beta \qquad \text{and} \qquad (\alpha + \beta)\pm \omega i, \,\,  -(\alpha-\beta)$$
respectively. The 1D-connections are given by: 
\begin{eqnarray*}
\overline{W^u(O_1)} \cap \EU^3&=& \overline{W^s(O_2)} \cap \EU^3=\text{Fix}(\mathbb{SO}(2)(\gamma_\psi))\cap \EU^3\\ 
&=& \{(x_1,x_2,x_3,x_4): x_1=x_2=0, x_3^2 + x_4^2 = 1\}
\end{eqnarray*}
and the 2D-connection is contained in
\begin{eqnarray*}
\overline{W^u(O_2)} \cap \EU^3&=& \overline{W^s(O_1)} \cap \EU^3=\text{Fix}(\ZZ_2(\gamma_2))\cap \EU^3 \\
&=& \{(x_1,x_2,x_3,x_4): x_1^2 + x_2^2 + x_4^2 = 1, x_3=0\} .
\end{eqnarray*}

The two-dimensional invariant manifolds of $O_1$ and $O_2$ are contained in the two-sphere $\text{Fix}(\ZZ_2(\gamma_2 ))\,\cap\, \EU^3.$ It is precisely the symmetry $\ZZ_2(\gamma_2)$ that forces the two-invariant manifolds $W^u(O_2)$ and $W^s(O_1)$ to coincide. We denote by $\Gamma$ the \emph{heteroclinic network} formed by the two equilibria, the two connections $[O_1 \rightarrow O_2]$ and the sphere $[O_2 \rightarrow O_1]$.   Keeping $\tau_1=\tau_2=0$, the equilibria $O_1$ and $O_2$ have the same \emph{chirality}. Therefore:
\begin{lemma}
If $\tau_1=\tau_2=0$,  the flow of \eqref{example} satisfies \textbf{(P1)--(P5)} described in Section \ref{starting point}. 
\end{lemma}
As a consequence, for $\tau_1=\tau_2=0$, the flow of \eqref{example} exhibits an asymptotically stable heteroclinic network $\Gamma$ associated to $O_1$ and $O_2$. The parameters $\tau_1$ and $\tau_2$   play the role of $A$ and $\lambda$, respectively, of \textbf{(P7)--(P8)}, after possible rescaling.

 \begin{corollary} \cite{CastroR2019}
\label{torus_cor}
For $\tau_1>0$ and $\tau_2=0$, close to the ``ghost'' of the attracting network $\Gamma$, the flow of \eqref{example}  has an attracting invariant two-torus, which is normally hyperbolic. 
\end{corollary}
    When $\tau_1\gg \tau_2>0$, although we break the $\mathbb{SO}(2)(\gamma_\psi)$--equivariance, the $\ZZ_2(\gamma_\pi)$--symmetry is preserved. This is why the connections lying in  $x_1=x_2=0$ persist.

\begin{lemma}\cite{CastroR2019}
For $\tau_1,\tau_2 > 0$ small enough such that $\tau_1\gg \tau_2$, the flow of \eqref{example} satisfies \textbf{(P7)--(P8b)}. 
\end{lemma}

Numerical simulations of  \eqref{example} for $\tau_1\gg \tau_2>0$  suggest the existence of regular and  chaotic behaviour in the region of transition from an attracting 2-dimensional torus to rotational horseshoes \cite{CastroR2019}. Chaotic attractors with  one positive Lyapunov exponent seem to exist, as suggested by the yellow regions  occurring in the upper part of the Arnold tongues in Figure \ref{numerics1}. The description of Section \ref{dissection} agrees quite well with the bifurcation diagram.

Hopf surfaces found in Corollary \ref{Corol1} correspond to the lower bound of the blue ``bananas'' that one observes in Figure \ref{numerics1}. This   bifurcation gives rise to a stable 2-torus (blue region) in the flow of \eqref{example}. Numerically we lose control of this stable torus, although we guess that it persists in other location of the phase space. 
Our numeric findings show that there are tiny regions of the parameter space inside the resonance regions where strange attractors may be found.


\subsection*{Technicalities on numerics of Figure \ref{numerics1} }
  The parameter plane $(\tau_1, \tau_2)$ of Figure \ref{numerics1} is scanned with a sufficiently small step along each coordinate axes. The software evaluates at each parameter value how many Lyapunov exponents along the orbit with initial condition $(0.1; 0.1; 0; -0.99)\in W^u(O_2)$, are non-negative (considered ``positive'' when greater than $  5\times 10^{-4}$ to discard uncertain positive Lyapunov exponents due to numerical precision issues). The parameter is painted according to the following rules: red for $0$, blue for $1$, yellow for $2$.  To estimate the complete Lyapunov spectra, the authors of \cite{CastroR2019} used the algorithm for differential
equations  with a Taylor series integrator.  

\begin{figure}[h]
\begin{center}
\includegraphics[height=7.9cm]{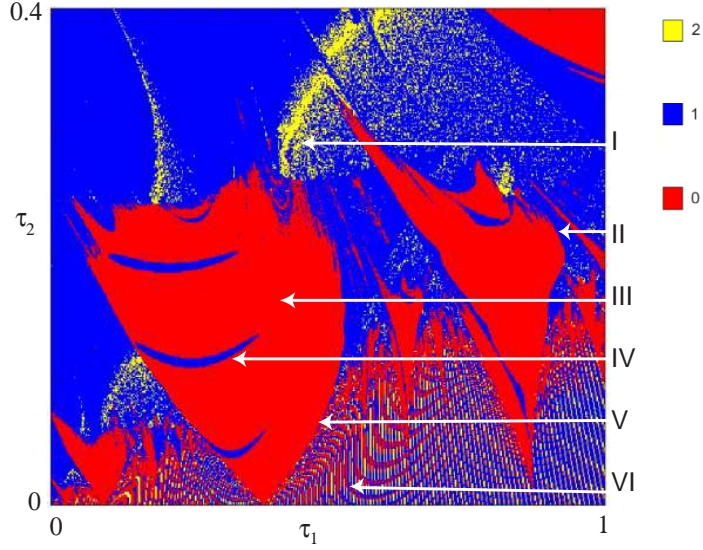}
\end{center}
\caption{\small Number of non-negative Lyapunov exponents along the orbit with initial condition $(0.1; 0.1; 0; -0.99)$   near $W^u(O_2)$ for equation \eqref{example} with $\alpha=1$, $\beta=-0.1$ and $\omega=1$ , $t\in [0,3750]$. Red for $0$; Blue for $1$; Yellow for $2$.  \textbf{I} -- {Homoclinic bifurcations; H\'enon-like strange attractors}; \textbf{II} -- {Sink}; \textbf{III}   -- {Resonant tongue (Arnold tongue)}; \textbf{IV} -- {Hopf bifurcation};  \textbf{V} -- {Saddle-node bifurcation (border of the Arnold tongue)}; \textbf{VI}  -- {Irrational torus (thin yellow region)}.  Figure  performed by L. Castro adapted from \cite{CastroR2019}.}
\label{numerics1}
\end{figure}

 \section{Discussion}
 \label{discussion}

In this article, we have constructed  a bifurcation diagram associated to a 3-parameter differential equation, whose starting point is  a weakly attracting heteroclinic network with a 2-dimensional connecting manifold, a natural configuration in symmetric systems and in some {unfoldings of the Hopf-zero singularity} \cite{BIS,  Langford, SNN95}.

We concentrate our attention in a family of vector fields $f_{\mu} \in\mathfrak{X}^3(\EU^3)$ satisfying \textbf{(P7)--(P8b)--(P9)}.  The   bifurcation diagram  of each element of the family is governed by   an \emph{Arnold wedge}, a structure through which an \emph{Arnold tongue} may be seen as a projection. This (new) heteroclinic bifurcation is different from that obtained in  \cite{Anishchenko},  in which an equilibrium produces a periodic solution which, in turn, generates a 2-torus. 

The structure of the Arnold tongue  strongly depends on $\omega$. 
This  suggested us to extend the 2-parameter ``classical'' bifurcation diagram $\left(A,  \frac{\lambda}{A}\right)$ of \cite{Rodrigues2019} to  a 3-dimen\-sional case, where $\omega$ is the additional parameter. 
Doing that,  \emph{Arnold tongues} give rise to  \emph{resonance wedges} bounded by two surfaces that correspond to saddle-node bifurcations. 
 The resonance wedge contains a sequence of curves corresponding to a discrete-time Bogdanov-Takens bifurcation, a
possibility  already anticipated in  \cite{Algaba2001, Kim}. Parameters within the wedge corresponds to maps whose periodic orbits share the same rotation number.

The structure of a \emph{resonance wedge} is consistent with the \emph{Torus-breakdown theory} \cite{AS91,  Ostlund, Peckman90, Peckman_bananas}, an essential route to understand the nature of turbulence  \cite{RT71}. When the speed of rotation $\omega>0$ is small, the flow of \eqref{general2.1} exhibits an attracting torus  consisting of either locked or quasiperiodic solutions. As $\omega$ increases, the attracting torus disintegrates into isolated periodic sinks and saddles. Increasing the magnitude of  $\omega$ further, the phase space is stretched and folded, creating rotational horseshoes, homoclinic tangencies and strange attractors of H\'enon-type  with  ergodic SRB measures.  In between,  Hopf bifurcations are present.  Our description refines   the diagrams proposed by  \cite{AP93} and  \cite{Aronson}.

For $f_{\mu} \in\mathfrak{X}^3(\EU^3)$, we may distinguish the dynamics between  \emph{heteroclinic tangle} and \emph{rank-one like attractors}; this depends on the hypotheses \textbf{(P8a)} and \textbf{(P8b)}.
In both cases, there exists complicated dynamics in $\mathcal{U}$,  but chaotic dynamics are created by two independent mechanisms. They are:

 \begin{eqnarray*}
\lambda>A \geq 0, \, \omega \in \RR^+    &\Leftrightarrow& \quad \text{\textbf{(P8a)}: $W^u(O_2)$ and $W^s(O_1)$ intersects transversely} \\ 
 &\Rightarrow& \text{the expansion induced by intersection of invariant manifolds}\\
  &\Rightarrow& \text{Smale horseshoes (heteroclinic tangles).}  
\end{eqnarray*} 

 \begin{eqnarray*}
A>\lambda \geq 0, \, \omega \in \RR^+  &\Leftrightarrow& \quad \text{\textbf{(P8b)}: $W^u(O_2)$ and $W^s(O_1)$ do not intersect} \\ 
 &\Rightarrow& \text{the invariant manifolds of the saddle-foci are pulled apart }\\
  &\Rightarrow& \text{the expansion is induced by large  $\omega$} \\
  &\Rightarrow& \text{Rotational horseshoes}.
\end{eqnarray*} 

In the first case, we conjecture the existence of a  \emph{non-uniform expansion} for a set with positive Lebesgue measure. The difficulties to prove the  conjecture are linked with the existence of infinitely many points within $W^s(O_1)$ where the first return map is not well defined. 
In the second scenario, the dynamics is governed by \emph{strange attractors} \cite{Wang2, WY}.
 A lot more needs to be done before these two types of chaos are well understood.

The analysis in this paper is not sensitive to the particular configuration given by the heteroclinic attractor $\Gamma$; the results are valid for more general weakly attracting networks with $2$-dimensional  heteroclinic connections which unfold generically from the coincidence. 
Finally, we would like to point out that all results also hold for periodically-forced differential equations, natural in the  study of \emph{seasonally forced systems}, where ``our''  parameters $A, \lambda, \omega$ may be interpreted  as (cf. \cite{LR2020}):
 \begin{eqnarray*}
\text{A} \quad &\to& \quad \text{Average of the periodic-forcing;}\\
{\lambda}\quad &\to& \quad \text{Effect (fluctuations) of the unstable manifold on a global cross section;}\\
{\omega}  \quad &\to& \quad \text{Frequency of the forcing}.
\end{eqnarray*}
By moving parameters, the invariant manifolds of invariant saddles cause destruction and fusion of attractors. The fully description of these metamorphoses is under analysis and are deferred for future work. 
 
\section*{Acknowledgements}
 The  author is grateful to Isabel Labouriau for the fruitful discussions during the research work performed in \cite{LR2020}. Special thanks to Andrey Shilnikov for pointing out the paper \cite{SNN95} on bifurcations analysis of a low-order atmospheric circulation model. The author is indebted to the two reviewers for the constructive comments, corrections and suggestions which helped to improve the readability of this manuscript.

\appendix

 \section{Glossary}
 \label{Definitions}
We record a miscellaneous collection of terms and terminology that are used throughout the text.
For $\varepsilon>0$ small, consider the 3-parameter family of $C^3$--smooth autonomous differential equations
\begin{equation}
\label{general2}
\dot{x}=f_{(A, \lambda, \omega)}(x)\qquad x\in \EU^3\subset\RR^4  \qquad A, \lambda \in [0, \varepsilon], \qquad \omega \in\RR^+.
\end{equation}
Since $\EU^3$ is a compact set without boundary, the local solutions of \eqref{general2} could be extended to $\RR$. Denote by $\varphi_{(A, \lambda, \omega)}(t,x)$, $t \in \RR$, the associated flow.

\subsection{Symmetry}
\label{app: symmetry}
Given a compact Lie group $\mathcal{G}$ of endomorphisms of $\EU^3$, we will consider 3-parameter families of vector fields $(f_{(A, \lambda, \omega)})$ under the equivariance assumption $$f_{(A, \lambda, \omega)}(\gamma x)=\gamma f_{(A, \lambda, \omega)}(x)$$ for all $x \in \EU^3$, $\gamma \in \mathcal{G}$ and $(A, \lambda, \omega )\in  [0, \varepsilon]^2\times \RR^+.$
For an isotropy subgroup $\widetilde{\mathcal{G}}< \mathcal{G}$, we write $\Fix(\widetilde{\mathcal{G}})$ for the vector subspace of points that are fixed by the elements of $\widetilde{\mathcal{G}}$. Note that, for $\mathcal{G}-$equivariant differential equations, the subspace $\Fix(\widetilde{\mathcal{G}})$ is flow-invariant.

\subsection{Attracting set}
\label{quasi1}
A subset $\Omega$ of  $ \EU^3$ for which there exists a neighborhood $U \subset  \EU^3$ satisfying $\varphi_{(A, \lambda, \omega)}(t,U)\subset U$ for all $t\geq 0$ and $$\dpt \bigcap_{t\,\in\,\RR^+}\,\varphi_{(A, \lambda, \omega)}(t,U)=\Omega$$ is called an \emph{attracting set} by the flow. This set is not necessarily connected. Its basin of attraction, denoted by $\textbf{B}(\Omega)$, is the set of points in $ \EU^3$ whose orbits have $\omega-$limit in $\Omega$. We say that $\Omega$ is \emph{asymptotically stable} (or  $\Omega$ is a \emph{global attractor}) if $\textbf{B}(\Omega)= \EU^3$. An attracting set is said to be \emph{quasi-stochastic} if it encloses periodic solutions with different Morse indices (dimension of the unstable manifold), structurally unstable cycles, sinks and saddle-type invariant sets (cf. \cite{Gonchenko97}).

\subsection{Heteroclinic structures}
\label{app: HSt}
Suppose that $O_1$ and $O_2$ are two hyperbolic equilibria of   \eqref{general2}  with different Morse indices (dimension of the unstable manifold). There is a {\em heteroclinic cycle} associated to $O_1$ and $O_2$ if
$$W^{u}(O_1)\cap W^{s}(O_2)\neq \emptyset \qquad \text{and} \qquad W^{u}(O_2)\cap W^{s}(O_1)\neq \emptyset.$$ For $ i\neq  j \in \{1,2\}$, the non-empty intersection of $W^{u}(O_i)$ with $W^{s}(O_j)$ is called a \emph{heteroclinic connection} between $O_i$ and $O_j$, and will be denoted by $[O_i \rightarrow  O_j]$. Although heteroclinic cycles involving equilibria are not a generic property within differential equations, they may be structurally stable within families of vector fields which are equivariant under the action of a compact Lie group $\mathcal{G}\subset \mathbb{O}(n)$, due to the existence of flow-invariant subspaces \cite{GH}.

\medbreak
A heteroclinic cycle between two hyperbolic saddle-foci of different Morse indices, where one of the connections is transverse while the other is structurally unstable, is called a \emph{Bykov cycle}. 
We address the reader to \cite{HS} for an overview of heteroclinic bifurcations and substantial information on the dynamics near different types of   structures.

\subsection{Historic behaviour}
\label{ss: historic behaviour}
We say that the solution of \eqref{general2}, $\varphi_{(A, \lambda, \omega)}(t,x)$ with $x \in \EU^3$, has \emph{historic behaviour} if 
there is a continuous function ${H}:\EU^3\rightarrow \RR $ such that
 the time average
$\dpt\frac{1}{T}\int_{0}^{T} {H} (\varphi_{(A, \lambda, \omega)}(t,x)) dt\, \, $
fails to converge.

\subsection{Strange attractor}
\label{ss: strange attractor}
A  (H\'enon-type) \emph{strange attractor} of a two-dimensional dissipative diffeomorphism $R$ defined in a Riemannian manifold  $\mathcal{M}$, is a compact invariant set $\Lambda$ with the following properties: 
\begin{itemize}
\item $\Lambda$  equals the closure of the unstable manifold of a hyperbolic periodic point;
\item the basin of attraction of $\Lambda$   contains an open set;
\item there is a dense orbit in $\Lambda$ with a positive Lyapounov exponent (exponential growth of the derivative along its orbit).
\end{itemize}
A vector field possesses a strange attractor if the first return map to a cross section does.  

\subsection{SRB measure}
\label{ss: SRB measure}
Given an attracting set ${\Omega}$ for a continuous map $R: \mathcal{M} \rightarrow \mathcal{M}$ where  $ \mathcal{M}$ is a compact smooth manifold, consider the Birkhoff average with respect to the continuous function $T:  \mathcal{M} \rightarrow \RR$ on the $R$-orbit starting at $x\in  \mathcal{M}$:
\begin{equation}
\label{limit1}
L(T, x)=\lim_{n\in \NN} \quad \frac{1}{n} \sum_{i=0}^{n-1} T \circ R^i(x).
\end{equation}

Suppose that, for Lebesgue almost all points $x\in \textbf{B}({\Omega})$,  the limit \eqref{limit1} exists and is independent on $x$. Then $L$ is a continuous linear functional in the set of continuous maps from  $\mathcal{M}$ to $\RR$ (denoted by $C(\mathcal{M}, \RR)$). By the \emph{Riesz Representation Theorem}, it defines a unique probability measure $\mu$ such that:
\begin{equation}
\label{limit2}
\lim_{n\in \NN} \quad \frac{1}{n} \sum_{i=0}^{n-1} T \circ R^i(x) = \int_{\Omega} T \, d\mu
\end{equation}
for all $T\in C(\mathcal{M}, \RR)$ and for Lebesgue almost all points $x\in \textbf{B}({\Omega})$.  If there exists an ergodic measure $\mu$ supported in ${\Omega}$ such that \eqref{limit2} is satisfied for all continuous maps $T\in C(\mathcal{M}, \RR)$ for Lebesgue almost all points $x\in \textbf{B}({\Omega})$, where $\textbf{B}({\Omega})$ has positive Lebesgue measure, then $\mu$ is called a SRB (Sinai-Ruelle-Bowen)  measure and ${\Omega}$ is a SRB attractor.  More details in \cite{WY}.

\subsection{Non-trivial wandering domains}
\label{ss: wandering}
A \emph{non-trivial wandering domain}  for a given map $R$ on a Riemannian manifold $ \mathcal{M}$ is a non-empty connected open set $D \subset  \mathcal{M}$ which satisfies the following conditions:
\medbreak
\begin{itemize} 
\item $R^i(D)\cap R^j(D)=\emptyset$ for every $i,j\geq 0$ ($i\neq j$) \\
\item the union of the $\omega$-limit sets of points in $D$ for $R$, denoted by $\Omega(D,R)$, is not equal to a single periodic orbit.
\medbreak
\end{itemize}
A wandering domain $D$ is called \emph{contracting} if the \emph{diameter} of $R^n(D)$ converges to zero as $n \rightarrow +\infty$.

\subsection{Rotational horseshoe}
\label{Rotational horseshoe}
Let $\mathcal{H} $ stand for the infinite annulus $\mathcal{H} = \EU^1 \times \RR$ (endowed with the usual inner product from $\RR^2$). We denote by $Homeo^+(\mathcal{H} )$ the set of homeomorphisms of the annulus which preserve orientation.
Given a homeomorphism $f :X \rightarrow X$  and a partition of $m\in \NN\backslash\{1\}$ elements $R_0,..., R_{m-1}$ of $X\subset \mathcal{H}$, the itinerary  function $\xi: X \rightarrow \{0, ..., m-1\}^\ZZ= \Sigma_m$ is defined by: $$\xi(x)(j)=k\quad   \Leftrightarrow \quad f^j(x)\in R_k, \quad \text{for every} \quad j\in \ZZ.$$  
Following \cite{PPS}, we say that a compact invariant set $\Lambda \subset \mathcal{H} $ of $f \in Homeo^+(\mathcal{H} )$ is a \emph{rotational horseshoe} if it admits a finite partition $P =\{R_0, ..., R_{m-1} \}$ by sets $R_i$ with non empty interior in $\Lambda$ so that:
\begin{itemize}
\item the itinerary $\xi$ defines a semi-conjugacy between $f|_\Lambda$ and the full-shift $\sigma: \Sigma_m \rightarrow \Sigma_m$, that is $\xi  \circ f = \sigma \circ \xi$ with $\xi$ continuous and onto;
\medbreak
\item for any lift $F: \RR^2 \rightarrow \RR^2$ of $f$, there exist  $k>0$ and $m$ vectors $v_0, ...,v_{m-1} \in \ZZ \times \{0\}$ so that:
$$
\left\| (F^n(\hat{x})-\hat{x})  - \sum_{i=0}^n v_{\xi(x)(i)}\right\| <k \qquad \text{for every} \qquad  \hat{x}\in \pi^{-1}(\Lambda), \quad n\in \NN,
$$
where $\|\star\|$ is the usual norm on $\RR^2$,  $\pi:\RR^2\rightarrow \mathcal{H}$ denotes the usual projection map and $\hat{x} \in \pi^{-1}(\Lambda)$ is the lift of $x$; more details in the proof of Lemma 3.1 of \cite{PPS}. The existence of a rotational horseshoe for a map implies positive topological entropy  at least  $ \log m $.
\end{itemize}

\end{document}